\documentclass[eqno] {amsart}
\usepackage{amscd}
\usepackage{amssymb}
\usepackage{cite}
\usepackage{amsmath}
\usepackage[usenames, dvipsnames]{color}

\usepackage{bookmark}
\bookmarksetup{level=section}
\usepackage{enumerate}

\emergencystretch=\maxdimen
\def\w*lim{\mathop{\mbox{\textup{w*-lim}}}}

\newtheorem{theorem}{\sc \textbf{Theorem}}[section]
\newtheorem{lemma}[theorem]{\sc \textbf{Lemma}}
\newtheorem{corollary}[theorem]{\sc \textbf{Corollary}}
\newtheorem{proposition}[theorem]{\sc \textbf{Proposition}}
\newtheorem{remark}[theorem]{\sc \textbf{Remark}}
\newtheorem{definition}[theorem]{\sc \textbf{Definition}}
\newtheorem{example}[theorem]{\sc \textbf{Example}}
\newtheorem{question}[theorem]{\sc Question}

\newcounter{cnt1}
\newcounter{cnt2}
\newcounter{cnt3}
\newcounter{cnt4}
\newcommand{\blr}{\begin{list}{$($\roman{cnt1}$)$} {\usecounter{cnt1}
 \setlength{\topsep}{0pt} \setlength{\itemsep}{0pt}}}
\newcommand{\blR}{\begin{list}{\Roman{cnt4}.\ } {\usecounter{cnt4}
 \setlength{\topsep}{0pt} \setlength{\itemsep}{0pt}}}
\newcommand{\bla}{\begin{list}{$($\alph{cnt2}$)$} {\usecounter{cnt2}
 \setlength{\topsep}{0pt} \setlength{\itemsep}{0pt}}}
\newcommand{\bln}{\begin{list}{$($\arabic{cnt3}$)$} {\usecounter{cnt3}
 \setlength{\topsep}{0pt} \setlength{\itemsep}{0pt}}}
\newcommand{\el}{\end{list}}

\newcommand{\real }{{\mbox{Re}}}
\newcommand{\imm }{{\mbox{Im}}}

\newcommand{\cE}{{\mathcal E}}
\newcommand{\cF}{{\mathcal F}}

\newcommand{\cH}{{\mathcal H}}

\newcommand{\cL}{{\mathcal L}}
\newcommand{\cM}{{\mathcal M}}
\newcommand{\cN}{{\mathcal N}}

\newcommand{\cP}{{\mathcal P}}

\newcommand{\cZ}{{\mathcal Z}}

\newcommand{\nn}[1]{{\left\vert\kern-0.25ex\left\vert\kern-0.25ex\left\vert #1
    \right\vert\kern-0.25ex\right\vert\kern-0.25ex\right\vert}}

\begin{document}
\title[Logarithmic submajorisation and order-preserving linear isometries]{Logarithmic submajorisation and order-preserving  linear isometries}
\date{}
{\renewcommand{\baselinestretch}{1}

\begin{abstract}
Let $\cE$ and $\cF$ be
noncommutative
 operator spaces   affiliated with semifinite von Neumann algebras $\cM_1$ and $\cM_2$, respectively.
We establish a noncommutative version of Abramovich's theorem \cite{A1983}, which provides the    general form of  normal order-preserving linear operators $T:\cE \stackrel{into}{\longrightarrow} \cF$ having the disjointness preserving property.
As an application, we obtain a noncommutative  Huijsmans-Wickstead theorem \cite{Huijsmans_W}.
By establishing  the disjointness preserving property  for an order-preserving isometry $T:\cE\rightarrow \cF$ from a noncommutative symmetrically $\Delta$-normed (in particular, quasi-normed) space into another,
we obtain the
existence of a Jordan $*$-monomorphism from $\cM_1$ into $\cM_2$ and the  general form of this isometry,
which extends and complements a number of existing results such as
 \cite[Theorem 1]{Broise}, \cite[Corollary 1]{Russo}, \cite[Theorem 2]{Sourour} and \cite[Theorem 3.1]{CMS}.
 In particular, we fully resolve the case when $\cF$ is the predual of $\cM_2$ and other untreated cases   in  \cite{SV}.
\end{abstract}
}

\author[J. Huang]{J. Huang}
\address[Jinghao Huang]{School of Mathematics and Statistics, University of New South Wales, Kensington, 2052, NSW, Australia  \emph{E-mail~:} {\tt
jinghao.huang@unsw.edu.au}}

\author[F. Sukochev]{F. Sukochev}
\address[Fedor Sukochev]{School of Mathematics and Statistics, University of New South Wales, Kensington, 2052, NSW, Australia  \emph{E-mail~:} {\tt f.sukochev@unsw.edu.au}
}

\author[D. Zanin]{D. Zanin}
\address[Dmitriy Zanin]{School of Mathematics and Statistics, University of New South Wales, Kensington, 2052, NSW, Australia  \emph{E-mail~:} {\tt d.zanin@unsw.edu.au}
}

\maketitle

\section{Introduction}

Let $\cM$ be a  von Neumann algebra
equipped with a faithful normal semifinite trace $\tau$.
The $*$-algebra $S(\mathcal{M},\tau)$ of all $\tau$-measurable operators affiliated with $\cM$
 is
fundamentally important  in noncommutative integration theory and/or in
(semifinite version of) noncommutative geometry because it  contains all $\mathcal{M}$-bimodules of interest in these fields.
Noncommutative $L_p$-spaces,  or,  more generally,   noncommutative
symmetric spaces,  associated with $\mathcal{M}$
 are solid subspaces in $S(\cM,\tau)$ \cite{LSZ,Tak},
which are equipped with unitarily invariant (quasi)-norms (or even $\Delta$-norms).
We view such bimodules
 as the noncommutative counterpart of the rearrangement invariant function spaces (see e.g. \cite{HM,Astashkin}) which are important  examples of partially  ordered topological vector spaces \cite{Jameson,Meyer_N}.
Indeed,  the real subspace $S_h(\cM,\tau)$ (respectively, $E_h(\cM,\tau)$) of $S(\cM,\tau)$ (respectively, a symmetrically $\Delta$-normed space $E(\cM,\tau)$) consisting of all self-adjoint elements in $S(\cM,\tau)$ (respectively, $E(\cM,\tau)$) is a partially ordered vector space.
Here, the partial ordering is an extension of the  natural  ordering in $
 \cM_h$, the real subspace of $\cM$ consisting of all self-adjoint operators.
The prime intention of this paper is to demonstrate that any order-preserving (or positive) linear  isometry  from one such bimodule into another
is generated by a Jordan $*$-monomorphism.
In particular,
let $\mathcal{E}=E(\mathcal{M}_1,\tau_1)$ and $\mathcal{F}=F(\mathcal{M}_2,\tau_2)$ be  symmetrically $\Delta$-normed operator spaces associated with semifinite von Neumann algebras  $(\mathcal{M}_1,\tau_1)$ and $(\cM_2,\tau_2)$, respectively.
We note that any lattice of measurable functions on the real line can be equipped with a $\Delta$-norm but not necessarily a norm.
The so-called  $\Delta$-normed operator spaces are the  natural noncommutative counterparts of lattices of measurable functions, and order-preserving isometries (or even general order-preserving linear operators) on  $\Delta$-normed operator spaces are the noncommutative counterparts of linear operators on lattices.
In this paper, we show that if  there exists an   order-preserving   linear isometry  $T:\cE\rightarrow \cF$ (i.e., $T(x)\ge 0$, $\forall 0\le x\in \cE$), then $\cM_1$ and a weakly closed $*$-subalgebra of $\cM_2$ are Jordan $*$-isomorphic.
Even though order-preserving isometries are  proper noncommutative counterparts of isometries between function spaces \cite{KR,FJ},
isometries of self-adjoint parts of symmetric operator spaces are  much harder to describe than those for  function spaces.


We shall omit the adjective ``linear" as we do not consider non-linear isometries in this paper.
The description of    isometries from one  noncommutative space into/onto  another
 has been widely studied since the 1950s \cite{Kadison51}.
In particular,
\mbox{Kadison \cite{Kadison51}} showed that a surjective isometry between two von Neumann algebras can be written as
 a Jordan $*$-isomorphism  multiplied by  a unitary operator,
which should be considered as a  noncommutative version of the  Banach-Stone Theorem \cite{Banach}.
After the non-commutative $L_p$-spaces were introduced in the 1950s \cite{Se},
the description of $L_p$-isometries was  investigated by Broise \cite{Broise}, Russo \cite{Russo}, Arazy \cite{Arazy} and Tam \cite{Tam}.
Finally, the complete description (for the semifinite case) was  obtained in 1981 by Yeadon \cite{Y}, i.e.,
every  isometry  $T:L_p(\cM_1,\tau_1) \stackrel{into}{\longrightarrow} L_p(\cM_2,\tau_2)$, $1\le p\ne 2<\infty$, is generated by  a Jordan $*$-isomorphism from $\cM_1$ onto a weakly closed $*$-subalgebra of $\cM_2$ (see \cite{Broise} for order-preserving isometries on noncommutative $L_2$-spaces, see also \cite{LZ}).
In the present paper, we concentrate on the following general  question.

\begin{question}\label{Q1}
Let  $\cE $ and $\cF $ be  symmetrically $\Delta$-normed operator spaces, respectively associated with semifinite von Neumann algebras $(\mathcal{M}_1,\tau_1)$ and $(\cM_2,\tau_2)$.
What is the general form of  the order-preserving  isometries $T$ from  $ \cE $ into $  \cF$?
i.e., is every   order-preserving  isometry $T: \cE \rightarrow   \cF$ generated by a Jordan $*$-homomorphism from $\cM_1$ into $\cM_2$?
\end{question}

When $\cM_1,\cM_2$ are finite von Neumann algebras and $\cE$ and $\cF$ are some special examples of Banach symmetric spaces,
surjective isometries have been widely studied (see e.g. \cite{MS1,MS2,Russo, CMS}).
The case for injective isometries  is substantially  more involved than for surjective isometries, which has been recently  treated in~\cite{SV}.
When $\cE$ is a (Banach) symmetric space and $\cF$ is a fully symmetric space having a strictly $K$-monotone norm (see Section \ref{s:p}),   all order-preserving isometries   $T:\cE \stackrel{into}{\longrightarrow} \cF$ are generated by Jordan $*$-homomorphisms from $\cM_1$ into $\cM_2$ \cite{SV}.
However, even the usual $L_1$-norm is not strictly $K$-monotone and this important case could not be  treated by techniques developed   in \cite{SV}.
Moreover, when   $\cM_1$ and $\cM_2$ are semifinite,
the unit element ${\bf 1}_{\cM_1}$ does not belong to any symmetric space affiliated with $\cM_1$ having order continuous norm when $\tau_1({\bf 1}_{\cM_1} )=\infty$ (see Remark \ref{remark:ocb}, see also \cite{DP2,DDP2,DPS}).
This fact
presents additional technical obstacles and
the description of order-preserving isometries $T:\cE \stackrel{into}{\longrightarrow} \cF$ in the general semifinite case  was left open  \cite[Section 5]{SV}.
In this   case, the general form of  isometries $T: \cE \stackrel{into}{\longrightarrow} \cF$ is obtained  only in the special  case
when   $\cE= L_p(\cM_1,\tau_1)$ and $\cF= L_p(\cM_2,\tau_2) $, $p> 0$, $p\ne 2$ (see e.g. \cite{Sherman2005,Y,Sherman,JRS}, see also \cite{S-1996,Sourour} for results for symmetric spaces affiliated with specific semifinite algebras).
One of the initial motivations of the present paper is to resolve the problem left in
\cite[Section 5]{SV} and to present new approaches which allow for study of order-preserving isometries of quasi-normed spaces and $\Delta$-normed spaces.
Our results are new even in the classical (commutative) setting (see e.g. \cite{Arazy1985,KR}), as we are able to treat injective isometries   between symmetrically quasi-normed, and even $\Delta$-normed, spaces,
which appear to be non-amenable to any previously used techniques mostly developed for Banach spaces setting.

To further elaborate this point, recall that  every symmetrically normed operator space is a subspace of $L_1(\cM,\tau)+\cM$ (see e.g. \cite{KPS,LSZ,DP2}).
The  $L_{\log}$-space $\cL_{\log}(\cM,\tau)$
 plays a similar role for  $\Delta$-normed/quasi-normed spaces as $L_1(\cM,\tau)$ does in the normed case,
 which was  recently introduced and studied in \cite{DSZ2015}.
 That is, the majority of symmetrically $\Delta$-normed spaces   used in analysis are subspaces of $\cL_{\log}(\cM,\tau)+\cM$.
In the present paper, we consider $\Delta$-normed operator spaces $\cF$  which are subspaces of   $ \cL_{\log}(\cM_2,\tau_2)+\cM_2$.

The main method used for the description of isometries
is to establish and employ the  ``disjointness preserving'' property, which
underlies all investigations in the general (noncommutative) case.
This idea lurks in the background of Yeadon's description \cite{Y} (see also \cite{Sherman2005,Sherman,Tam})
of isometries of noncommutative $L^p$-spaces ($1\le p\ne 2 <\infty$),
whose  proof  relies on the study when we have the equality in the Clarkson's inequality.
However,
Abramovich \cite[Remark 2, p.78]{A1} emphasised that
order-preserving isometries from a (Banach) symmetric   space $\mathcal{E}$ into another (Banach) symmetric   space
$\mathcal{F}$ may not necessarily enjoy the  ``disjointness preserving'' property,  even in the commutative setting.
Still,
in the commutative setting, the
``disjointness preserving'' property of order-preserving isometries can be
 guaranteed by the so-called {\it strict monotonicity} of the norm $\left\|\cdot\right\|_\cF$.
 That is,  by the assumption that $0\leq z_1<z_2\in \cF$, we have  $\left\|z_1\right\|_\cF< \left\|z_2 \right\|_\cF$.
 It is natural to
  consider the following question in order to answer   Question \ref{Q1}.
 \begin{question}
Let the  conditions of Question \ref{Q1} hold.
Does  $T$ preserves disjointness? That is, does the equality $T(x)T(y)=0$ hold whenever $xy=0$, $0\le x,y \in \cE$?
\end{question}
The so-called {\it strictly $K$-monotone norms} (see \cite{SV} or Section \ref{s:p})
 form a proper noncommutative counterpart to  the   notion of strictly monotone norms, which were introduced in \cite{ S-1992, CDSS, DSS} as an important component in the characterisation of Kadec-Klee type properties.
 Using the \lq\lq the triangle inequality for the  Hardy-Littlewood preorder\rq\rq\,  introduced in \cite{CS1, CMS} and analysing when this inequality turns into equality,
 it is shown in \cite{SV} that every order-preserving isometry into $\cF$ possesses the ``disjointness preserving'' property whenever $\cF\subset (L_1+L_\infty)(\cM_2,\tau_2) $ is a symmetric space  with   strictly $K$-monotone  norm $\left\|\cdot\right\|_\cF$.
The drawback of this approach is the fact that many important symmetric norms fail to be  strictly K-monotone. In particular, as mentioned before, even the usual $L_1$-norm is not strictly $K$-monotone.
To rectify this drawback and cover maximally wide class of (quasi-normed and $\Delta$-normed) symmetric spaces,
 we introduce the   notion of \emph{strictly log-monotone (SLM) $\Delta$-norms} (see Section \ref{s:p}), which should be considered as a far-reaching generalisation of the strict $K$-monotonicity.
The class of SLM $\Delta$-norms embraces an  extensive class of symmetric $\Delta$-norms.
For example,  the usual $L_p$-norms ($0<p<\infty$),  Lorentz quasi-norms and the log-integrable-$F$-norm $\left\|\cdot\right\|_{\log}$,   are  examples of SLM $\Delta$-norms (for which we refer to  Section \ref{s:p} and Section \ref{s:l}).
Using techniques developed from  detailed study of logarithmic submajorisation, we show   that  every order-preserving  isometry $T:\cE \stackrel{into} {\longrightarrow} \cF$ possesses the  \lq\lq disjointness preserving\rq\rq\ property if  $\left\|\cdot \right\|_\cF$ is an  SLM $\Delta$-norm.
Surprisingly, this result appears to be new even for symmetrically $\Delta$-normed function spaces.

With \lq\lq disjointness preserving\rq\rq\ property at hand,  we establish a general description of all   order-preserving injective  isometries $T$ as above,
showing that every such isometry is  generated by a Jordan $*$-isomorphism from $\cM_1$ onto a weakly closed $*$-subalgebra of $\cM_2$.
The description of disjointness preserving operators on Banach lattices has been well-studied (see \cite{AK1,A1983}, see also \cite{Arendt,Araujo,LW}).
In particular, Abramovich \cite{A1983} obtained the general description of order-continuous (or normal) disjointness-preserving operators on banach lattices.
However, due to the lack of structure of lattices, the case for disjointness-preserving operators on  noncommutative operator spaces is much harder than that on lattices.
As far as we know, there is no literature on this theme.
One of the  main results   of the present paper is a noncommutative version of Abramovich's theorem \cite{A1983} (see also \cite{AK1}), which allows us to describe the general form of order-preserving isometries.
As a consequence, we establish a noncommutative version of  Huijsmans-Wickstead theorem \cite{Huijsmans_W}.

This result extends and strengthens  a number of existing results in the literature.
On the one hand, we extend \cite[Proposition 6]{SV} (see also \cite{Russo, Katavolos2}) to the case of arbitrary semifinite von Neumann algebras.
 On the other hand,  we establish that the main results of \cite{SV, CMS, Broise,Katavolos2,Russo} continue to hold in a much wider setting than in those papers.
In particular,  we resolve   the $L_1$-case  which was not amenable to the techniques based on strict $K$-monotonicity used in  \cite{SV}.
When $\cM_2$ is a semifinite factor,
we obtain  a semifinite version of \cite[Corollary 1]{Russo},
showing
 that every  order-preserving  isometry $T:\cE \stackrel{onto}{\longrightarrow} \cF$ coincides with      a   $*$-isomorphism or a $*$-anti-isomorphism  multiplied  by a positive constant  on $\cE \cap \cM_1$.
In particular,
when $\cM_1 $ and $\cM_2$ are finite factors,
$  T|_{\cM_1}$  is indeed   a   $*$-isomorphism or a $*$-anti-isomorphism from $\cM_1$ onto $\cM_2$ multiplied  a positive constant,
  which recovers  and substantially extends \cite[Corollary 1]{Russo}.

\section{Preliminaries}\label{s:p}

In this section, we recall main notions of the theory of noncommutative integration, introduce some properties of generalised singular value functions and define noncommutative symmetrically $\Delta$-normed spaces.
In what follows,  $\cH$ is a Hilbert space and $B(\cH)$ is the
$*$-algebra of all bounded linear operators on $\cH$, and
$\mathbf{1}$ is the identity operator on $\cH$.
Let $\mathcal{M}$ be
a von Neumann algebra on $\cH$.
For details on von Neumann algebra
theory, the reader is referred to e.g. \cite{Dixmier}
or \cite{Tak}. General facts concerning measurable operators may
be found in \cite{Nelson}, \cite{Se} (see also the forthcoming book \cite{DPS}).
For convenience of the reader, some of the basic definitions are recalled.

\subsection{$\tau$-measurable operators and generalised singular values}

A linear operator $x:\mathfrak{D}\left( x\right) \rightarrow \cH $,
where the domain $\mathfrak{D}\left( x\right) $ of $x$ is a linear
subspace of $\cH$, is said to be {\it affiliated} with $\mathcal{M}$
if $yx \subseteq xy$ for all $y\in \mathcal{M}^{\prime }$, where $\mathcal{M}^{\prime }$ is the commutant of $\mathcal{M}$.
A linear
operator $x:\mathfrak{D}\left( x\right) \rightarrow \cH $ is termed
{\it measurable} with respect to $\mathcal{M}$ if $x$ is closed,
densely defined, affiliated with $\mathcal{M}$ and there exists a
sequence $\left\{ p_n\right\}_{n=1}^{\infty}$ in the logic of all
projections of $\mathcal{M}$, $\cP\left(\mathcal{M}\right)$, such
that $p_n\uparrow \mathbf{1}$, $p_n(\cH)\subseteq\mathfrak{D}\left(x \right) $
and $\mathbf{1}-p_n$ is a finite projection (with respect to $\mathcal{M}$)
for all $n$.
It should be noted that the condition $p _{n}\left(
\cH\right) \subseteq \mathfrak{D}\left( x\right) $ implies that
$xp _{n}\in \mathcal{M}$. The collection of all measurable
operators with respect to $\mathcal{M}$ is denoted by $S\left(
\mathcal{M} \right) $, which is a unital $\ast $-algebra
with respect to strong sums and products (denoted simply by $x+y$ and $xy$ for all $x,y\in S\left( \mathcal{M}\right) $).

Let $x$ be a self-adjoint operator affiliated with $\mathcal{M}$.
We denote its spectral measure by $\{e^x\}$.
It is well known that if
$x$ is a closed operator affiliated with $\mathcal{M}$ with the
polar decomposition $x = u|x|$, then $u\in\mathcal{M}$ and $e\in
\mathcal{M}$ for all projections $e\in \{e^{|x|}\}$.
Moreover,
$x \in S(\mathcal{M})$ if and only if $x $ is closed,
 densely
defined, affiliated with $\mathcal{M}$ and $e^{|x|}(\lambda,
\infty)$ is a finite projection for some $\lambda> 0$.
 It follows
immediately that in the case when $\mathcal{M}$ is a von Neumann
algebra of type $III$ or a type $I$ factor, we have
$S(\mathcal{M})= \mathcal{M}$.
For type $II$ von Neumann algebras,
this is no longer true.
From now on, let $\mathcal{M}$ be a
semifinite von Neumann algebra equipped with a faithful normal
semifinite trace $\tau$.

For any closed and densely defined linear operator $x :\mathfrak{D}\left( x \right) \rightarrow \cH $,
the \emph{null projection} $n(x)=n(|x|)$ is the projection onto its kernel $\mbox{Ker} (x)$,
 the \emph{range projection } $r(x )$ is the projection onto the closure of its range $\mbox{Ran}(x)$ and the \emph{support projection} $s(x)$ of $x$ is defined by $s(x) ={\bf{1}} - n(x)$.

An operator $x\in S\left( \mathcal{M}\right) $ is called $\tau$-measurable if there exists a sequence
$\left\{p_n\right\}_{n=1}^{\infty}$ in $\cP \left(\mathcal{M}\right)$ such that
$p_n\uparrow \mathbf{1}$, $p_n\left( \cH \right)\subseteq \mathfrak{D}\left(x \right)$ and
$\tau(\mathbf{1}-p_n)<\infty $ for all $n$.
The collection of all $\tau $-measurable
operators is a unital $\ast $-subalgebra of $S\left(
\mathcal{M}\right) $,  denoted by $S\left( \mathcal{M}, \tau\right)
$.
It is well known that a linear operator $x$ belongs to $S\left(
\mathcal{M}, \tau\right) $ if and only if $x\in S(\mathcal{M})$
and there exists $\lambda>0$ such that $\tau(e^{|x|}(\lambda,
\infty))<\infty$.
Alternatively, an unbounded operator $x$
affiliated with $\mathcal{M}$ is  $\tau$-measurable (see
\cite{FK}) if and only if
$$\tau\left(e^{|x|}(n,\infty)\right)\rightarrow 0,\quad n\to\infty.$$
For any $x=x^*\in S\left( \mathcal{M}, \tau\right)$, we set $x_+=xe^{x}(0,\infty)$ and $x_-=xe^{x}(-\infty,0)$.

\begin{definition}
Let a semifinite von Neumann  algebra $\mathcal M$ be equipped
with a faithful normal semi-finite trace $\tau$ and let $x\in
S(\mathcal{M},\tau)$. The generalised singular value function $\mu(x):t\rightarrow \mu(t;x)$, $t>0$,  of
the operator $x$ is defined by setting
$$
\mu(t;x)
=
\inf\{\|xp\|:\ p=p^*\in\mathcal{M}\mbox{ is a projection,}\ \tau(\mathbf{1}-p)\leq t\}.
$$
\end{definition}
An equivalent definition in terms of the
distribution function of the operator $x$ is the following. For every self-adjoint
operator $x\in S(\mathcal{M},\tau) $,  setting
$$d_x(t)=\tau(e^{x}(t,\infty)),\quad t>0,$$
we have (see e.g. \cite{FK} and \cite{LSZ})
$$
\mu(t; x)=\inf\{s\geq0:\ d_{|x|}(s)\leq t\}.
$$
Note that $d_x(\cdot)$ is a right-continuous function (see e.g.  \cite{FK} and \cite{DPS}).

Consider the algebra $\mathcal{M}=L^\infty(0,\infty)$ of all
Lebesgue measurable essentially bounded functions on $(0,\infty)$.
Algebra $\mathcal{M}$ can be seen as an abelian von Neumann
algebra acting via multiplication on the Hilbert space
$\mathcal{H}=L^2(0,\infty)$, with the trace given by integration
with respect to Lebesgue measure $m.$
It is easy to see that the
algebra of all $\tau$-measurable operators
affiliated with $\mathcal{M}$ can be identified with
the subalgebra $S(0,\infty)$ of the algebra of Lebesgue measurable functions which consists of all functions $f$ such that
$m(\{|f|>s\})$ is finite for some $s>0$. It should also be pointed out that the
generalised singular value function $\mu(f)$ is precisely the
decreasing rearrangement $\mu(f)$ of the function $|f|$ (see e.g. \cite{KPS}) defined by
$$\mu(t;f)=\inf\{s\geq0:\ m(\{|f|\geq s\})\leq t\}.$$

For convenience of the reader,  we also recall the definition of the \emph{measure topology} $t_\tau$ on the algebra $S(\cM,\tau)$. For every $\varepsilon,\delta>0,$ we define the set
$$V(\varepsilon,\delta)=\{x\in S(\mathcal{M},\tau):\ \exists p \in \cP\left(\mathcal{M}\right)\mbox{ such that }
\left\|x(\mathbf{1}-p)\right\|_\infty \leq\varepsilon,\ \tau(p)\leq\delta\}.$$ The topology
generated by the sets $V(\varepsilon,\delta)$,
$\varepsilon,\delta>0,$ is called the measure topology $t_\tau$ on $S(\cM,\tau)$ \cite{DPS, FK, Nelson}.
It is well known that the algebra $S(\cM,\tau)$ equipped with the measure topology is a complete metrizable topological algebra \cite{Nelson}.
We note that a sequence $\{x_n\}_{n=1}^\infty\subset S(\cM,\tau)$ converges to zero with respect to measure topology $t_\tau$ if and only if $\tau\big( e  ^{|x_n|}(\varepsilon,\infty)\big)\to 0$ as $n\to \infty$ for all $\varepsilon>0$ \cite{DPS}.

The space  $S_0(\cM,\tau)$ of $\tau$-compact operators is the space associated to the algebra of functions from $S(0,\infty)$ vanishing at infinity, that is,
$$S_0(\cM,\tau) = \{x\in S(\cM,\tau) :  \ \mu(\infty; x) =0\}.$$
The two-sided ideal $\cF(\tau)$ in $\cM$ consisting of all elements of $\tau$-finite range is defined by
$$\cF(\tau)=\{x\in \cM ~:~ \tau(r(x)) <\infty\} = \{x \in \cM ~:~ \tau(s(x)) <\infty\}.$$
Clearly, $S_0(\cM,\tau)$ is the closure of $\cF(\tau)$ with respect to the measure topology \cite{DP2}.

A further important vector space topology on $S(\cM,\tau)$ is the \emph{local measure topology} \cite{DP2,DPS}.
A neighbourhood base for this topology is given by the sets $V(\varepsilon, \delta; p )$, $\varepsilon, \delta>0$, $p\in \cP(\cM)\cap \cF(\tau)$, where
$$V(\varepsilon,\delta;  p ) = \{x\in S(\cM,\tau): pxp \in V(\varepsilon,\delta)\}. $$
It is clear that the local measure topology is weaker than the measure topology~\cite{DP2,DPS}.
If $\{x_\alpha\}\subset S(\cM,\tau)$ is a net and if $x_\alpha \rightarrow_\alpha x \in S(\cM,\tau)$ in local measure topology, then $x_\alpha y\rightarrow xy $ and $yx _\alpha \rightarrow yx $ in the local measure topology for all $y \in S(\cM,\tau)$ \cite{DP2,DPS}.\label{lmtc}

\subsection{Symmetrically $\Delta$-normed  spaces of $\tau$-measurable operators}
For convenience of the reader, we recall the definition of $\Delta$-norms.
Let $\Omega$ be a linear space over the field $\mathbb{C}$.
A function $\left\|\cdot\right\|$ from $\Omega$ to $\mathbb{R}$ is a $\Delta$-norm, if for all $x,y \in \Omega$ the following properties hold:
\begin{align}
\left\|x\right\| \geqslant 0 , ~\left\|x\right\| = 0 \Leftrightarrow x=0 ;\\
\label{def:1}\left\|\alpha x\right\| \leqslant \left\|x\right\|, ~\forall~  |\alpha| \le 1 ;\\
\label{de:d:3}\lim _{\alpha \rightarrow 0}\left\|\alpha x\right\| = 0;\\
\label{def:d:4}\left\|x+y \right\| \le C_\Omega \cdot (\left\|x\right\|+\left\|y\right\|)
\end{align}
 for a constant $C_\Omega\geq 1$ independent of $x,y$.
The couple $(\Omega, \left\|\cdot\right\|)$ is called a \emph{$\Delta$-normed} space.
We note that the definition of a $\Delta$-norm given above is the same with that given in \cite{KPR}.
It is well-known that every $\Delta$-normed space $(\Omega,\left\|\cdot\right\|)$ is metrizable  and conversely every metrizable space can be equipped with a $\Delta$-norm  \cite{KPR}.
Note that properties $(2)$ and $(4)$ of a $\Delta$-norm imply that for any $\alpha\in\mathbb{C}$, there exists a constant $M$ such that $\|\alpha x\|\leq M\|x\|,\, x\in \Omega$, in particular, if $\|x_n\|\to 0, \{x_n\}_{n=1}^\infty\subset \Omega$, then $\|\alpha x_n\|\to 0$.
In particular, when $C_\Omega=1$, $\Omega$ is called an \emph{$F$-normed} space \cite{KPR}.
%

Let $E(0,\infty)$  be a space of real-valued Lebesgue measurable
functions on  $(0,\infty)$ (with identification
$m$-a.e.), equipped with a $\Delta$-norm $\left\|\cdot\right\|_E$.
The space $E(0,\infty)$ is said to be {\it
absolutely solid} if $x\in E(0,\infty)$ and $|y|\leq |x|$, $y\in S(0,\infty)$
implies that $y\in E(0,\infty)$ and $\|y\|_E\leq\|x\|_E.$
An absolutely solid space $E(0,\infty)\subseteq S(0,\infty)$ is said to be {\it
symmetric} if for every $x\in E(0,\infty)$ and every $y\in S(0,\infty)$,
 the assumption
$\mu(y)=\mu(x)$ implies that $y\in E(0,\infty)$ and $\|y\|_E=\|x\|_E$ (see e.g.
\cite{KPS,Astashkin}).

We now come to the definition of the main object of this paper.
\begin{definition}\label{opspace}
Let $\cM $ be a semifinite von Neumann  algebra  equipped
with a faithful normal semi-finite trace $\tau$.
Let $\mathcal{E}$ be a linear subset in $S({\mathcal{M}, \tau})$
equipped with a $\Delta$-norm $\|\cdot\|_{\mathcal{E}}$.
We say that
$\mathcal{E}$ is a \textit{symmetrically $\Delta$-normed  space}  if
for $x \in
\mathcal{E}$, $y\in S({\mathcal{M}, \tau})$ and  $\mu(y)\leq \mu(x)$ imply that $y\in \mathcal{E}$ and
$\|y\|_\mathcal{E}\leq \|x\|_\mathcal{E}$.
\end{definition}
Let $E(\cM,\tau)$ be a symmetrically $\Delta$-normed space.
Since $\mu(axb)\le \mu(\|a\|_\infty \|b\|_\infty x) $, $a,b\in \cM$, $x\in E(\cM,\tau)$, it follows that
every symmetrically $\Delta$-normed space is an $\cM$-bimodule.
It is well-known that any symmetrically normed space $E(\cM,\tau)$ is a normed $\cM$-bimodule (see e.g.  \cite{DP2} and \cite{DPS}). 
However, one should note that a symmetrically  $\Delta$-normed space $E(\cM,\tau)$ does not necessarily satisfy
$\|axb\|_E \le \|a\|_\infty \|b\|_\infty \|x\|_E, ~a, b \in \cM,~ x\in E(\cM,\tau)$.
 For every  $x\in E(\cM,\tau)$ and $\{y_n\in \cM\}$ with $\|y_n\|_\infty \rightarrow 0$,
  $\mu(xy_n),\mu(y_n x)\le \|y_n\|_\infty\mu(x ) = \mu(\|y_n\|_\infty x )$ implies that
  $ xy_n ,y_n x  \in E(\cM,\tau) $ and
 \begin{align}\label{ineq:to0}
\left\|xy_n\right\|_E, \left\|y_n x \right\|_E\le   \Big\|\|y_n \|_\infty x \Big\|_E\stackrel{\eqref{de:d:3}}{\rightarrow} 0,
 \end{align}
Definition \ref{opspace} together with \cite[Lemma 2.3.12 and Corollary 2.3.17]{LSZ} implies that
\begin{align}\label{eq:*|}
\left\|x\right\|_E = \left\|x ^*\right\|_E = \left\||x|\right\|_E,~x\in E.
\end{align}


There exists a strong connection between symmetric function spaces and
operator spaces exposed in \cite{Kalton_S} (see also \cite{Sukochev, HLS2017, Bik1992}).
The operator space $E(\cM,\tau)$ defined by
\begin{equation*}
E(\mathcal{M},\tau):=\{x \in S(\mathcal{M},\tau):\ \mu(x )\in E(0,\infty)\},
\ \left\|x \right\|_{E(\mathcal{M},\tau)}:=\left\|\mu(x )\right\|_E
\end{equation*}
 is a complete symmetrically $\Delta$-normed space  whenever $(E(0,\infty),\left\|\cdot\right\|_E)$ is    a complete  symmetrically $\Delta$-normed function space on $(0,\infty)$  \cite{HLS2017} (see also \cite{Kalton_S,Sukochev}).

For  a  given   symmetrically $\Delta$-normed space $E(\cM,\tau)$,
we denote $\cP(E) : = E(\cM,\tau)\cap \cP(\cM)$.
If $p,q\in \cP(E)$, then $p\vee q \in \cP(E)$ (see e.g.  \cite[Chapter IV,  Lemma 1.4]{DPS} or \cite[Lemma 4]{DP2}).
The \emph{carrier projection} $c_E\in \cM$ of an $\cM$-bimodule $E$ is defined by setting
$$c_E := \vee \{p:p\in \cP(E)\}.$$
 It is clear that $c_E$ is in the center of $\cM$ \cite{DP2}.
The following proposition is an extension of \cite[Chapter IV, Lemma 4.4]{DPS}.
\begin{proposition}\label{prop:car}
If the carrier projection $c_E$ of a symmetrically $\Delta$-normed space $E(
\cM,\tau)$ is equal to ${\bf 1}$, then
$$\{p\in \cP(\cM):\tau(p)<\infty\}\subset \cP(E)$$
and hence, $\cF(\tau)\subset E(\cM,\tau)$.
\end{proposition}
\begin{proof}
By \cite[Lemma 4 (iii)]{DP2}, the set $\cP(E)$ is upwards directed and the normality of trace $\tau$ implies that
\begin{align} \label{eq:sup1}
\sup\{ \tau(p) : P\in \cP(E)\} =\tau({\bf 1}).
\end{align}

Suppose first that $\tau({\bf 1})=\infty$.
If $q\in \cP(\cM)$ satisfies $\tau(q)<\infty$, then \eqref{eq:sup1} implies that $\tau(q)\le \tau(p)$ for some $p\in \cP(E)$ and hence, $q\in \cP(E)$.
This proves the assertion in the case that $\tau({\bf 1})=\infty$.

Assume that $\tau({\bf 1})<\infty$.
It suffices to show that ${\bf 1} \in E(\cM,\tau)$.
It follows from \eqref{eq:sup1}  that there exists $p\in \cP(E)$ such that $\tau(p)\ge \frac12 \tau({\bf 1})$.
Since $\tau(p^\perp)\le \frac12 \tau({\bf 1}) \le \tau(p)$, it follows that also $p^\perp  \in \cP(E)$ and hence, ${\bf 1}\in \cP(E)$.

If $x\in \cF(\tau)$, then the support projection $p=s(x)$ of $x$ satisfies $\tau(p)<\infty$ and $|x| \le \left\|x\right\|_\infty p$.
That is, $\mu(x)=\mu(|x|)\le \mu(\left\|x\right\|_\infty p )$.
Since $p\in \cP(E)$ and $E(\cM,\tau)$ is a linear space, it follows Definition \ref{opspace} that $x\in E(\cM,\tau)$.
\end{proof}
It is often assumed that the carrier projection $c_E$ is equal to ${\bf 1}$.
Indeed, for any  symmetrically $\Delta$-normed function space  $E(0,\infty)$ on the interval $(0,\infty)$, the carrier projection of the corresponding operator space $E(\cM,\tau)$ is always ${\bf 1}$ (see e.g.
\cite{HS}, see also \cite{HM,Astashkin}).
In the present paper, we always assume that the carrier projection of a symmetrically $\Delta$-normed space is equal to ${\bf 1}$.

\subsection{Submajorisation}

If $x,y\in S(\cM,\tau)$, then $x$ is said to be submajorised by $y$, denoted by $x \prec\prec y$ (Hardy-Littlewood-Polya submajorisation), if
\begin{align*}
\int_{0}^{t} \mu(s;x) ds \le \int_{0}^{t} \mu(s;y) ds
\end{align*}
for all $t\ge 0$ (see e.g. \cite{LSZ,DPS,DP2}).
The algebra
$$ \cL_{\log}(\cM,\tau):=\{x \in S(\cM,\tau): \|x \|_{\log}:=\int_0^\infty \log(1+\mu(t;x ))dt < \infty\}$$
 of log-integrable operators introduced in
\cite{DSZ2015} is a complete symmetrically $F$-normed space.
Denote  $\log_+ t :=\max \{\log t, 0\}$.
For $ x,y \in S(\cM,\tau)$ with $\log_+ \mu(x),\log_+ \mu(y)\in L_1(0,\infty)+L_\infty(0,\infty)$, $x$ is said to be logarithmically submajorised by $y$ \cite{Hiai,DDSZ}, denoted by $x\prec\prec_{\log} y$, if
\begin{align*}
\int_{0}^{t} \log(\mu(s;x) )ds \le \int_{0}^{t}\log( \mu(s;y)) ds, ~t\ge 0.
\end{align*}
In particular,   we have
$
\mu(xy)\prec\prec_{\log} \mu(x)\mu(y)$ (see \cite[Theorem 1.18]{Hiai} or \cite{DDSZ}).

For the sake of convenience, we denote $\cM^\Delta:= (\cL_{\log}(\cM,\tau)+\cM)\cap S_0(\cM,\tau)$.
In particular, for $x\in S(\cM,\tau)$, $x\in \cM^\Delta$ if and only if  $\log_+\mu(x ) \in L_1(0,\infty)+L_\infty(0,\infty)$ and $\mu(\infty;x) =0$.

A (Banach) symmetric norm $\left\|\cdot\right\|_E$ on $E(\cM,\tau)$ is called strictly $K$-monotone if and only if $\|x\|_E < \|y\|_E$ whenever $x,y\in E(\cM,\tau)$, $x\prec \prec y$ and $\mu(x)\ne \mu(y)$. 
It is natural to introduce the following notion when considering symmetrically $\Delta$-normed (or quasi-normed) operator space.
A symmetric $\Delta$-norm on $E(\cM,\tau)\subset (\cL_{\log}+\cL_\infty )(\cM,\tau)$ is called a \emph{strictly log-monotone} (SLM) $\Delta$-norm if   $\|x\|_E < \|y\|_E$ whenever $x,y\in E(\cM,\tau)$ satisfies $\mu(x) \prec \prec _{\log}\mu(y)$ and $\mu(x)\ne \mu(y)$ ($\left\|\cdot\right\|_E$ is called \emph{log-monotone} if $\|X\|_E\le \|y\|_E$ whenever $\mu(x) \prec \prec _{\log}\mu(y)$).
Indeed, the usual $L_p$-norm $\left\|\cdot\right\|_p$, $0<p<\infty$, are SLM $\Delta$-norms.
In the last section, we show that   noncommutative Lorentz spaces associated with  $\cM$ are SLM quasi-normed.
It is clear that $E(\cM,\tau)$ has  SLM $\Delta$-norm whenever $\|\cdot\|_E$ is an SLM $\Delta$-norm on $E(0,\infty)$.

We denote the decreasing rearrangement $f^*$ of a measurable function $f$   by
$$f^*(t)=\inf\{s\geq 0:\ m(\{f\geq s\})\leq t\}.$$
The following result is well-known, which is essentially  an inequality of Hardy, Littlewood and Polya (see  \cite[Chapter 1, Theorem D.2]{MOA} for  results which imply the following, or
 \cite[Lemma]{Weyl} and \cite[Chapter II, Lemma 3.4]{GK1} for the sequence version).

\begin{proposition}\label{Prop:lefunction2}
Assume that  $f=f^* $ and $g=g^* $ are measurable function  integrable  on $(0,s)$, $s>0$.
If $\int_0^b  f(t)  dt \le  \int_0^b g(t  )dt $ for every $0 <b\le s $,
then for every increasing  continuous convex    function $\varphi$ on $\mathbb{R}$, we have $\int_0^b \varphi(f(t)) dt \le  \int_0^b \varphi(g(t ))dt $ for every $0 <b\le s $.
\end{proposition}


The following corollary is an easy consequence of Proposition \ref{Prop:lefunction2}.
\begin{corollary}\label{prop:e}
Let $x,y\in \cM^\Delta$.
If $\mu(x) \prec\prec_{\log} \mu(y)$, then $\mu(x)^p \prec\prec \mu(y)^p $, $0<p<\infty$.
\end{corollary}
\begin{proposition}
$\left\|\cdot \right\|_{\log}$  is an SLM symmetric $F$-norm on $ \cL_{\log}(\cM,\tau)$.
\end{proposition}
\begin{proof}
Since $\|\cdot\|_{\log}$ is a symmetric $F$-norm on $ \cL_{\log}(\cM,\tau)$ \cite{DSZ2015}, it suffices to prove the SLM property.
Assume that  $x,y\in \cL_{\log}(\cM,\tau)$ with  $ \mu(x )\prec\prec_{\log}   \mu(y) $.
Without loss of generality, we may assume that $\mu(t;x)>0$ for every $t>0$ (therefore, $\mu(t;y)>0$, $t>0$).
Note that $\log(\mu(x) \chi_{(0,t)}), \log(\mu(y)\chi_{(0,t)})  $ are integrable functions on $(0,t)$.
Since $\mu(x)\prec \prec_{\log} \mu(y)$, it follows that
\begin{align*}
\mu(x)^\frac12 \chi_{(0,t)}\prec\prec_{\log} \mu(y)^\frac12  \chi_{(0,t)},
\end{align*}
and therefore,
$$
\mu(x)  \chi_{(0,t)}   \prec\prec_{\log} \mu(y)^\frac12 \mu(x)^\frac12  \chi_{(0,t)} \prec\prec_{\log}  \mu(y) \chi_{(0,t)}  .
$$
Taking a  continuous convex  function $\varphi(t):=\log(1+e^t)$, $t\in \mathbb{R}$,
by Proposition \ref{Prop:lefunction2},
we obtain that
\begin{align}\label{major+1}
\left(\mu(x)+1  \right)\chi_{(0,t)} \prec\prec_{\log } \left( \mu(y)^\frac12 \mu(x)^\frac12   +1 \right)\chi_{(0,t)} \prec\prec_{\log} \left(   \mu(y) +1 \right)\chi_{(0,t)}
\end{align}
for every $t>0$, which implies that $\|\cdot\|_{\log}$ is  log-monotone.
In addition, if  $\|x\|_{\log}=\|y\|_{\log}$, then \eqref{major+1} implies that
\begin{align}\label{log=}
&\qquad 2 \int_0^\infty \log(\mu(t;y)^\frac12 \mu(t;x)^\frac12 +1 )dt \nonumber\\
&=\int_0^\infty \log(\mu(t;x)+1 ) +\log(\mu(t;y) +1 )dt .
\end{align}
Since $(\mu(t;y)^\frac12 \mu(t;x)^\frac12 +1 )^2\le (\mu(t;x)+1 )(\mu(t;y) +1 ) $ and the equality holds true only when $\mu(t;x)=\mu(t;y)$,
it follows from \eqref{log=} that
$\mu(x)=\mu(y)$, a.e..
By the right-continuity of $\mu(x)$ and $\mu(y)$, we obtain that $\mu(x)=\mu(y)$.
Therefore, $\left\|\cdot \right\|_{\log}$ is an SLM $F$-norm on $ \cL_{\log}(\cM,\tau)$.
\end{proof}

\subsection{Order continuous $\Delta$-norms}

In this subsection, we introduce the notion of order continuous $\Delta$-norms.
For the introduction of order continuous norms, we refer to \cite{DDP2,DP2,DPS}.

If $E(\cM,\tau)\subset S(\cM,\tau)$ is a symmetrically $\Delta$-normed operator space, then the $\Delta$-norm $\|\cdot\|_E$ is called order continuous if $\|x_\alpha \|_E \rightarrow_\alpha 0$  whenever $\{x_\alpha\}$ is a downwards directed net in $E(\cM,\tau)^+$ satisfying $x_\alpha \downarrow 0$.
See e.g. \cite{DP2,DPS,KM} for examples of order continuous $\Delta$-norms.

The set of all self-adjoint elements of $E(\cM,\tau)$ is denoted by $E_h(\cM,\tau)$.
By \cite[Lemma 2.4]{HLS2017} and  \cite[Chapter II, Proposition 6.1]{DPS} (see also \cite[Corollary 4.3]{HS} and \cite[Proposition 2]{DP2}), we obtain  the following result immediately.
\begin{lemma}\label{cor:cone}
Let $E(\cM,\tau)$ is a symmetrically $\Delta$-normed space.
If $\{x_\lambda\}$ is an increasing  net in $E_h(\cM,\tau)$ and $x\in E_h(\cM,\tau)$ with $\|x_\lambda -x\|_E\rightarrow 0$, then $x_\lambda \uparrow x$.
\end{lemma}

\begin{proposition}
If $E(\cM,\tau)\subset S(\cM,\tau)$ is a symmetrically $\Delta$-normed space, then the following statements are equivalent:
\begin{itemize}
  \item[(i)] $E(\cM,\tau)$ has order continuous $\Delta$-norm;
  \item[(ii)] $\left\|x_n\right\|_E\downarrow 0$ for every decreasing sequence $\{x_n\}_{n=1}^\infty$ in $E(\cM,\tau)^+$ satisfying $x_n\downarrow 0$.
\end{itemize}
\end{proposition}
\begin{proof}
Since it is clear that (ii) follows from (i), it suffices to show that statement (ii) implies that $\left\|\cdot\right\|_E$ is order continuous.

Suppose that $\{x_\alpha\}$ is a decreasing net in $E(\cM,\tau)^+$ satisfying $x_\alpha \downarrow_\alpha 0$. It should be observed that this implies that $\{x_\alpha\}$ is a Cauchy net for the $\Delta$-norm.
Indeed, if $\{x_\alpha\}$ is not Cauchy, then there exists an $\varepsilon >0$ and an decreasing subsequence $\{x_{\alpha_n}\}_{n=1}^\infty$ such that $\|x_{\alpha_n} -x_{\alpha_{n+1}}\|_E \ge \varepsilon$ for all $n$.
By \cite[Proposition 2 (ii)]{DP2}, we obtain that there exists $y\in S(\cM,\tau)^+$ such that $x_{\alpha_n} \downarrow _n y$.
By assumption (ii),
this implies that $\|x_{\alpha_n} -y\|_E \rightarrow _n 0$.
Hence, we obtain that
$$\varepsilon \le \|x_{\alpha_n} -x_{\alpha_{n+1}}\|_E \le C_E(\|x_{\alpha_n} -y\|_E+\|x_{\alpha_{n+1}} -y\|_E )\rightarrow_n 0 ,$$
 which is a contradiction.
  This implies that there exists a decreasing subsequence $\{x_{\alpha_n}\}$ such that
  \begin{align}\label{ineq1n}
  \|x_{\alpha_n}-x_{\alpha }\|_E \le 1/n\end{align}
   for every $\alpha \ge \alpha_n$.
 Let $x\in S(\cM,\tau)^+$ be such that $x_{\alpha_n}\downarrow _n x$ (see \cite[Proposition 2]{DP2}).
 Since $x_{\alpha_n}\ge x$, it follows that $x\in E(\cM,\tau)^+$.
 It follows from (ii) that $\|x-x _{\alpha_n}\|_E \rightarrow 0$ as $n\rightarrow \infty$ and hence, by \eqref{ineq1n}, we have  $\|x-x_\alpha\|_E \rightarrow _\alpha 0$.
Appealing to Lemma \ref{cor:cone}, we obtain that that  $x_\alpha \downarrow x$.
 Hence, $x =0$ and so, $\|x_\alpha\|_E\downarrow_\alpha 0$.
\end{proof}
Assume that  $E(\cM,\tau)$ is a symmetrically $\Delta$-normed space.
The subset $E(\cM,\tau)^{oc}\subset E(\cM,\tau) $ is defined by setting
$$E(\cM,\tau)^{oc} = \{x \in E : |x|\ge x _\alpha \downarrow_\alpha 0\Rightarrow \|x_\alpha\|_E \downarrow 0 \} .$$
\begin{remark}\label{remark:ocb}
$E(\cM,\tau)^{oc}$ is a subspace of the $\left\|\cdot\right\|_E$-closure of $\cF(\tau)$ in $E(\cM,\tau)$.
Indeed, if $0\le x \in E(\cM,\tau)^{oc}$, then there exists an upwards directed net $\{x_\alpha\}$ in $\cF(\tau)^+$ such that $0\le x _\alpha \uparrow_\alpha x $ (see e.g. \cite[Corollary 8 (vi)]{DP2} or \cite[Chapter IV, Corollary 1.9]{DPS}), that is, $x \ge x-x_\alpha \downarrow_\alpha  0$.
Hence, $\|x-x_\alpha\|_E \downarrow_\alpha 0$.

Since $S_0(\cM,\tau)$ is closed in $S(\cM,\tau)$ with respect to the measure topology (see \cite[Section 2.4]{DP2})  and the embedding of $E(\cM,\tau)$ into $S(\cM,\tau)$ is continuous with respect to the measure topology (see \cite[Lemma 2.4]{HLS2017}), it follows from $\cF(\tau)\subset S_0(\cM,\tau)$ that  $E(\cM,\tau)^{oc}\subset S_0(\cM,\tau)$.
\end{remark}

\begin{proposition}\label{prop:ococ}
Let $E(0,\infty)\subset S(0,\infty)$ be a symmetrically $\Delta$-normed function space.
If $x\in E(\cM,\tau)$ and $\mu(x )\in E(0,\infty)^{oc}$, then $x\in E(\cM,\tau)^{oc}$.
In particular,
if $E(0,\infty)$
 has  order continuous $\Delta$-norm $\left\|\cdot \right\|_{E}$,
then $\left\|\cdot \right\|_{E}$ is an order continuous $\Delta$-norm on $E(\cM,\tau)$.
\end{proposition}
\begin{proof}
If $\mu(x)\in E(0,\infty)^{oc}$, then $x\in S_0(\cM,\tau)$  (see Remark \ref{remark:ocb}), that is, $\lim_{t\rightarrow \infty}\mu(t;x)=0$.
Suppose that $\{x_\alpha\}$ is a net in $E(\cM,\tau)$ such that $|x| \ge x_\alpha \downarrow _\alpha 0$.
It follows from \cite[Lemma  3.5]{DDP2} (see also \cite[Chapter III, Lemma 2.14]{DPS}) that $\mu(x)\ge \mu(x_\alpha)\downarrow _\alpha 0$ in $E(0,\infty)$.
Since $\mu(x)\in E(0,\infty)^{oc}$, this implies that $\|\mu(x_ \alpha)\|_E \downarrow_\alpha 0$,
that is, $\|x_\alpha\|_{E}\downarrow_\alpha 0$.
\end{proof}
We obtain the following corollary immediately. 
\begin{corollary}
$\left\|\cdot\right\|_{\log}$ is an order continuous $F$-norm on $\cL_{\log}(\cM,\tau)$.
\end{corollary}

Let $(\cM_1,\tau_1)$ and $(\cM_2,\tau_2)$ be two semifinite von Neumann algebras.
The general form of order-preserving  isometries  $T:\cL_{\log}(\cM_1,\tau_1) \stackrel{into}{\longrightarrow} \cL_{\log}(\cM_2,\tau_2)$ is obtained in Corollary \ref{cor:4.5}.
In particular, there is a Jordan $*$-homomorphism from $\cM_1$ into $\cM_2$.
We note that it is proved in  \cite{ACM} that in the special case when $(\cM_1,\tau_1)$ and $(\cM_2,\tau_2)$ are finite measure spaces,  all isometries  (not necessarily order-preserving)  $T: \cL_{\log}(\cM_1,\tau_1) \stackrel{into}{\longrightarrow} \cL_{\log}(\cM_2,\tau_2)$ are automatically disjointness preserving, which yields the general form of these isometries.
\subsection{Jordan $*$-isomorphism}
Let  $(\cM_1,\tau_1)$ and $(\cM_2,\tau_2)$ be two semifinite von Neumann algebras.
A  complex-linear map $J:\cM_1 \stackrel{into}{\longrightarrow} \cM_2$ is called Jordan $*$-homomorphism if  $J(x^*)=J(x)^*$ and $J(x^2)=J(x)^2$, $x\in \cM_1$ (equivalently, $J(xy+yx)=J(x)J(y)+J(y)J(x)$, $x,y\in \cM_1$).
The following definitions vary slightly in different literature.
In the present paper, we stick to the following definitions.
We call $J $ a Jordan $*$-monomorphism if it is injective.
If $J$ is a bijective
Jordan $*$-homomorphism, then it is called a Jordan $*$-isomorphism (see \cite[Definition 3.2.1]{BR}).
A   Jordan $*$-homomorphism is called normal if it is completely additive (equivalently, ultraweakly continuous).
Alternatively, we adopt the following equivalent definition: $J(x_\alpha)\uparrow J(x)$ whenever $x_\alpha\uparrow x\in \cM_1^+$ (see e.g \cite[Chapter I.4.3]{Dixmier}).
We note that there are some literature in which  injective Jordan $*$-homomorphisms are called Jordan $*$-isomorphisms, and bijective
Jordan $*$-homomorphisms are called surjective (or onto) Jordan $*$-isomorphisms (see e.g. \cite{Y,Sherman2005}).

For details on
Jordan $*$-homomorphism,
the reader is referred to \cite{BR} or \cite{Dixmier} (see also \cite{Kadison51} and \cite{Stormer}).
For the sake of convenience, we collect some properties of Jordan $*$-homomorphism/isomomorphism.
The next result is very simple and well-known (see e.g. \cite[p. 12]{SV}) and we omit its proof.
\begin{proposition}\label{prop:+-J}
Assume that $J:\cM_1\rightarrow \cM_2$ is a
 complex-linear positive (i.e., $J(a)\ge 0$, $a\in \cM_1^+$) or self-adjoint (i.e., $J(a)=J(a)^*$, $a=a^*\in \cM_1$)  mapping.
   If $J$
    satisfies that $J(x^2)=J(x)^2$ for every $x\in \cM_1^+$, then $J$ is a Jordan $*$-homomorphism.
\end{proposition}

The following result is fundamental in the study of Jordan $*$-homomorphism (see \cite[Theorem 3.3]{Stormer}, see also \cite{Kadison51} or \cite[Appendix]{RR}).
\begin{lemma}\label{2.12}If $J$ is a Jordan $*$-homomorphism from $\cM_1$ into $\cM_2$,
then there exists a projection
 $z$  in the center of the ultra-weak closure of $J(\cM_1)$ such that  $J(\cdot)z$ is a $*$-homomorphism and $J(\cdot)({\bf 1}_{\cM_2}-z)$ is a  $*$-anti-homomorphism on $\cM_1$.
\end{lemma}

\begin{proposition}\label{JXJE}
If $J$ is a Jordan $*$-homomorphism from $\cM_1$ into $\cM_2$, then
for any commuting $ x,y\in \cM_1$, we have
\begin{align}\label{eq27}
J(xy)=J(x)J(y)=J(y)J(x).
\end{align}
\end{proposition}
 \begin{proof}
Let $z$ be a projection in the center of the ultra-weak closure of $J(\cM_1)$ such that  $J(\cdot)z$ is a $*$-homomorphism and $J(\cdot)({\bf 1}_{\cM_2} -z)$ is a  $*$-anti-homomorphism on $\cM_1$ (see Lemma \ref{2.12}).
It now  follows that
\begin{align*}
J(xy) &= J(xy)z +J( xy)({\bf 1}_{\cM_2}-z)= J(xy)z +J(yx)({\bf 1}_{\cM_2}-z)\\
 &=J(x)J(y)z+J(x)J(y) ({\bf 1}_{\cM_2}-z) =J(x)J(y),
\end{align*}
which completes the proof.
\end{proof}

If  $J$ is a Jordan $*$-homomorphism, then  for any  self-adjoint $a\in \cM_1$,    $J(a)$ is self-adjoint.
It is well-known (see e.g. \cite{Stormer} or \cite[Page 211]{BR}) that every Jordan $*$-homomorphism is positive, i.e.,  if $a\ge 0$, then
\begin{align}\label{JA>0}
J(a)  \ge 0.\end{align}

The following proposition provides a criterion for verifying  that a Jordan $*$-homomorphism is injective.

\begin{proposition}\label{prop:fai}
If $J:\cM_1\rightarrow \cM_2$ is a  Jordan $*$-homomorphism, then $J$ is injective  if and only if $J(p) >0$ for every  $\tau_1$-finite $0\ne p\in \cP(\cM_1)$.
\end{proposition}
\begin{proof}
It
 suffices to prove the ``if'' part.

For every $x >0$, there exists a $\tau_1$-finite projection $p\in \cM_1$ such that $x\ge \lambda p$ for some $\lambda > 0$.
Therefore, $$J(x)=J(x-\lambda p) +J(\lambda p ) \stackrel{\eqref{JA>0}}{\ge }  J(\lambda p ) >0 .$$

Assume  that  $x \in \cM_1$ with $J(x )=0$.
Since $J$ is a Jordan $*$-homomorphism, it follows that  $ J(\real ( x))  $ and $ J( \imm (x))$ are self-adjoint.
Therefore,
by $ J(\real (x))  + i J( \imm (x))  = J(x) =0$, we obtain that  $ J(\real (x)) = J(\imm (x))=0$.

Let  $a :=  \real (x )$ or $\imm (x )$.
In particular,  $a $ is self-adjoint  with   $J(a ) =0$.
Then, $J(a _+) =J(a _-)\ge 0 $.
By \eqref{eq27}, we have $J(a _+)  J(a _-)=0$, which implies that $J(a _+)=  J(a_-)=0$,
i.e.,   $a_+=a_- =0$.
Hence, $a=0$. That is, $x=0$.
\end{proof}

\begin{remark}\label{isotoJM}
Assume that $J:\cM_1\rightarrow \cM_2$ is a normal Jordan $*$-homomorphism.
Clearly, $J(\cM_1)\subset J({\bf 1}_{\cM_1}) \cM_2 J({\bf 1}_{\cM_1})$ (see Proposition \ref{JXJE}).
By \cite[Part I, Chapter 4.3, Corollary 2]{Dixmier},
 $J(\cM_1)$ is  a weakly closed $*$-subalgebra of $\cM_2$.
In particular, if $J$ is an injective, then $J$ is a normal Jordan $*$-isomorphism from $\cM_1 $ onto $J(\cM_1)$.
\end{remark}

For the sake of convenience, we denote by
$\cP_{fin} (\cM_1) $ the subset of $\cP(\cM_1)$ whose elements have finite traces.
Lemma \ref{familyofF} below is drawn from Yeadon's proof \cite{Y},
which contains  a beautiful trick of constructing a Jordan $*$-homomorphism from Jordan $*$-homomorphisms on reduced von Neumann algebras $e\cM_1 e$, $e \in \cP_{fin} (\cM_1)$.
Before proceeding to the proof, let us recall the following fact concerning on strong operator convergence.
\begin{proposition}\label{prop:aeasyfact}
Let $\{x_i\}$ be a uniformly bounded net in a von Neumann algebra $\cM$ and $\{p_i\}$ be a increasing net of projections increasing to ${\bf 1}$.
If $x_i =p_i x_jp _i$ for every $j\ge i$, then $so-\lim_i x_i$ exists (denoted by $x$).
In particular, $x_i =p_i x p_i$.
\end{proposition}
\begin{proof}
By the compactness of the unit ball of a von Neumann algebra in the weak operator topology (see e.g. \cite[Chapter IX, Proposition  5.5]{Conway}),
there exists a wo-converging subnet $\{x_{i_k}\}$ of $\{x_i\}$.
Let $x:=wo-\lim_{i_k}x_{i_k}$.
In particular, we have $x_{i_k}=wo-\lim_{i_j\ge i_k} p_{i_k} x_{i_j } p_{i_k} = p_{i_k} xp_{i_k}$.
By the assumption, for every $i\le i_k$, we have
$x_i=p_i x_{i_k} p_i=p_i p_{i_k} xp_{i_k} p_i=p_i xp _i$.
Since $\{x_{i_k}\}$ is a subnet of $\{x_i\}$,
it follows that $x_i=p_i x p_i$ for every $i$.
Clearly, $p_ixp_i\rightarrow x $ in strong operator topology.
\end{proof}

\begin{lemma}\label{familyofF}
Let $\{J_e: e \cM_1 e \rightarrow \cM_2 \}_{E\in \cP_{fin} (\cM_1)}$ be a family of normal Jordan $*$-homomorphisms.
If for every $e \le f\in \cP_{fin} (\cM_1)$, we have $J_f=J_e$ on $e \cM_1 e $,
then there exists a normal Jordan $*$-homomorphism $J:\cM_1\rightarrow \cM_2$ agreeing with $J_e $ on $e\cM_1 e$ for every
 $e \in \cP_{fin} (\cM_1) $.
Moreover, if $J_e$ is injective for every $e\in \cP_{fin}(\cM_1)$, then $J$ is a normal Jordan $*$-isomorphism from $\cM_1 $ onto $J(\cM_1)$ and $J(\cM_1)$ is a weakly closed $*$-subalgebra  in $\cM_2$.
\end{lemma}
\begin{proof}
Note that $J_f(f)-J_e(e)=J_f(f-e) \ge 0$, $e\le f\in \cP_{fin}(\cM_1)$.
We define $$J({\bf 1}_{\cM_1}):= \sup \{J_f(f):f\in \cP_{fin}(\cM_1) \}.$$

Since $J_f$ is a Jordan $*$-homomorphism on $f \cM_1  f$, it follows from  \cite[Lemma 1]{Kaplansky} (or \cite[Lemma 2]{Herstein}) that  for every $x\in \cM_1$ and $e\le f\in \cP_{fin}(\cM_1)$, we have
\begin{align}
\label{stronglimit}
J_f(e)J_f(fxf) J_f(e)=  J_f(exe)=  J_e(exe).
\end{align}
Note that we have $\|J_e(eye)\|_\infty \le \left\|J_e(e\left\|y\right\|_\infty e )
\right\|_\infty\le \|y\|_\infty $, $0\le y\in \cM_1^+$.

Applying  Proposition \ref{prop:aeasyfact} to the reduced algebra $J({\bf 1}_{\cM_1})\cM_1 J({\bf 1}_{\cM_1})$,
for every $x\in \cM_1$,
 we get that $\{J_e(exe)\}_{e\in \cP_{fin}(\cM_1)}$ is a uniformly  bounded net converging in the strong operator topology,
where the net is indexed by the upwards-directed set of projections $e \in \cP_{fin}(\cM_1)$.
 So,
 we can extend $J$ to the whole  $\cM_1$ by defining
 \begin{align}\label{def:J}
 J(x):=so-\lim _{e \in \cP_{fin}(\cM_1)}  J_e(exe ).
 \end{align}
Moreover, by \eqref{stronglimit}, we obtain  that
\begin{align}\label{J}
J(e)J(x)J(e)=J(exe) =J_e(exe)
\end{align}
for every $x \in \cM_1$ and projection  $e \in \cP_{fin}(\cM_1)$.
By the definition,  $J$ is complex-linear and $J(x)$ is self-adjoint for every  self-adjoint $x \in \cM_1$.
It remains to prove that $J$ is a normal  Jordan $*$-homomorphism.

Given a net $\{x_\alpha\}$ with $ x_\alpha \uparrow x \in \cM_1^+$  and any $e \in \cP_{fin}(\cM_1) $,
by   \cite[Proposition 1 (vi)]{DP2} and the normality of $J_e$ on $e \cM_1 e$, $e \in \cP_{fin}(\cM_1)$, we obtain  that
\begin{align*}
J(e)J(x)J(e) \stackrel{\eqref{J}}{=} J(exe) =\sup_\alpha  J(ex_\alpha e)& \stackrel{\eqref{J}}{=} \sup_\alpha J(e)J(x_\alpha) J(e)\\
&= J(e)\sup_\alpha  J(x_\alpha) J(e).
\end{align*}
Now,
taking the (so)-limit,
we obtain that
$$J({\bf 1}_{\cM_1}) J(x)J({\bf 1}_{\cM_1}) =J({\bf 1}_{\cM_1}) \sup_\alpha  J(x_\alpha) J({\bf 1}_{\cM_1}) .$$
By the construction of $J$ (see \eqref{def:J} and \eqref{J}),  one see that  the support projections $s(J(x)), s(J(x_\alpha))  \le   \sup \{J(f):f\in \cP_{fin}(\cM_1)\}= J({\bf 1}_{\cM_1}) $, which implies  that  $J(x)=\sup J(x_\alpha)$.
That is,  $J$ is normal.

For any  $x\in \cM_1^+$, by the construction of $J$ (see \eqref{def:J} and \eqref{J}),  it is easy to see that   $ so-\lim_{f\in \cP_{fin}(\cM_1)} J(f)= J({\bf 1}_{\cM_1})\ge  s(J(x))$.
Hence, for any   $x \in \cM_1^+$, by the normality of $J$ and the (so)-continuity of multiplication on the unit ball of a von Neumann algebra,  we have
\begin{align*}
J(x^2) & = so-\lim_{e \in \cP_{fin}(\cM_1) } J(xex)= so-\lim_{e \in   \cP_{fin}(\cM_1)} J(e ) J(xex) J(e) \\
& \stackrel{\eqref{J}}{=}so-\lim_{e \in  \cP_{fin}(\cM_1)} J_e (exexe) =  so- \lim_{e\in  \cP_{fin}(\cM_1)} (J_e(exe))^2 \stackrel{\eqref{def:J}}{=}J(x)^2.
\end{align*}
By Proposition \ref{prop:+-J}, $J$ is a Jordan $*$-homomorphism.

Assume now that  $J_e$ is injective for every $e\in \cP_{fin}(\cM_1)$.
By Proposition \ref{prop:fai}, $J$ is a Jordan $*$-monomorphism.
Moreover, since $J$ is normal, it follows from
 Remark \ref{isotoJM} that $J(\cM_1 )$ is  a weakly closed $*$-subalgebra  in $\cM_2$.
\end{proof}


If $J$ is a normal Jordan $*$-monomorphism,
then  there exists a projection $p\in \cZ(\cM_1)$ (the center of $\cM_1$) such that $J$ is a $*$-monomorphism on $p\cM_1 p$ and a $*$-anti-monomorphism on $ ({\bf 1}_{\cM_1}-p)\cM_1 ({\bf 1}_{\cM_1}-p)$.
The following property is an easy consequence of this fact.
\begin{proposition}Assume that $J:\cM_1\rightarrow \cM_2$ is a normal  Jordan $*$-monomorphism.
There exists a projection $p \in \cZ(\cM_1)$  such that
for every $x \in \cM_1$, we have
\begin{align}\label{lemma:desofabs}
|J(x)| = J\left(p|x| + ({\bf 1}_{\cM_1}-p)|x^*| \right).
\end{align}
\end{proposition}
\begin{proof}

Let  $p\in \cZ(\cM_1)$ be  a projection such that $J$ is a $*$-monomorphism on $p\cM_1 p$ and a $*$-anti-monomorphism on $ ({\bf 1}_{\cM_1}-p)\cM_1 ({\bf 1}_{\cM_1}-p)$ (see e.g. \cite[Theorem 3.3]{Stormer}, see also \cite[p. 45]{Y}).
By \eqref{eq27},  $J(p)\perp J({\bf 1}_{\cM_1}-p)$.
It then follows that
for every $x\in \cM_1$, we have
 \begin{align*}
|J(x)|^2 &~= J(x)^* J(x) = (J(px)^* +J(({\bf 1}_{\cM_1}-p)x)^* )(J(px) +J(({\bf 1}_{\cM_1}-p)x)) \\
&\stackrel{\eqref{eq27}}{=} J(px)^*J(px)   +J(({\bf 1}_{\cM_1}-p)x)^*J(({\bf 1}_{\cM_1} -p)x)\\
&~=J(px^*x)   +J(({\bf 1}_{\cM_1}-p)xx^*)\\
&~=J(p|x|^2)   +J(({\bf 1}_{\cM_1}-p)|x^*|^2)\\
&\stackrel{\eqref{eq27}}{=} J(p|x| + ({\bf 1}_{\cM_1}-p)|x^*| )J(p|x| + ({\bf 1}_{\cM_1}-p)|x^*| ),
\end{align*}
which together with \eqref{JA>0} implies the validity of  \eqref{lemma:desofabs}.
\end{proof}

Jordan $*$-homomorphisms have strong connections with projection ortho-morphsims, that is,   mappings $\varphi:\cP(\cM_1)\rightarrow \cP(\cM_2)$ satisfying that $
\varphi(p)\perp \varphi(q)$ and $\varphi(p\vee q) =\varphi(p) +\varphi(q)$ for all mutually  orthogonal $p,q\in \cP(\cM_1)$ (see e.g. \cite{Dye}).
Whenever $\cM_1$ has no type $I_2$ direct summands, every the projection ortho-morphism from $\cP(\cM_1)$ into $\cP(\cM_2)$ can be extended to a Jordan $*$-homomorphism from $\cM_1$ into $\cM_2$ (see
\cite[Theorem 8.1.1]{Hamhalter}, see also \cite{Dye}).
In general, one can not expect that every ortho-morphism can be extended to a Jordan $*$-homomorphism  \cite[P.83]{Dye}.

\begin{lemma}\label{linear orthomorphism extension}
 Let $\varphi:\cM_1 \rightarrow \cM_2$ be a linear mapping which is continuous in the uniform norm topology.
 If the reduction of $\varphi$ on $\cP(
 \cM_1)$ is an ortho-homomorphism from  $\cP(\cM_1)$ into $\cP(\cM_2)$,
 then $\varphi$ is a Jordan $*$-homomorphism from  $\cM_1$ into $\cM_2$.
\end{lemma}
\begin{proof}
Let $0\le x\in \cM_1$.
Without loss of generality, we may assume that $\|x\|_\infty = 1$.
Define $x_n:=\sum_{k=1}^n \frac{k-1}{n} e^x(\frac{k-1}{n},\frac{k}{n}]$.
In particular, $\|x-x_n\|_\infty \le \frac1n$.

By linearity, we have
$$\varphi(x_n^2)=\varphi(\sum_{k=1}^n (\frac{k-1}{n})^2 e^x(\frac{k-1}{n},\frac{k}{n}])=\sum_{k=1}^n (\frac{k-1}{n})^2  \varphi( e^x(\frac{k-1}{n},\frac{k}{n}]) .$$
Since $\varphi|_{\cP(\cM_1)}$ is an ortho-homomorphism from $\cP(\cM_1)$ into $\cP(\cM_2)$, it follows that
$$\left(\varphi(x_n)\right)^2 = \left(\sum_{k=1}^n \frac{k-1}{n}  \varphi( e^x(\frac{k-1}{n},\frac{k}{n}]) \right)^2 = \sum_{k=1}^n (\frac{k-1}{n})^2  \varphi( e^x(\frac{k-1}{n},\frac{k}{n}])=\varphi(x_n^2) .$$
Since $\|x_n-x\|_\infty \rightarrow _n 0$, it follows from  the $\left\|\cdot\right\|_\infty$-continuity of $\varphi$ that $\varphi(x^2) =\left(\varphi(x)\right)^2$.
Moreover, since $\varphi(x_n)\ge 0$, it follows from the $\left\|\cdot\right\|_\infty$-continuity of $\varphi$ that $\varphi(x)\ge 0$.
By  Proposition \ref{prop:+-J}, $J$ is a Jordan $*$-homomorphism.
\end{proof}

\section{A noncommutative Abramovich's theorem}
The main object of this section is to establish a noncommutative version of Abramovich's theorem \cite{A1983}.
 Instead of considering isometries only (see e.g. \cite{Y,SV}), we provide a description to more general operators, i.e.,    order-preserving linear  operators which  are order-continuous (or normal) and disjointness preserving.
In particular, we show that if such a  operator is a bijection (in particular, it is  a \emph{Lamperti operator} \cite{Arendt}), then it is  a \emph{d-isomorphism } (i.e., both $T$ and $T^{-1}$ preserve disjointness \cite{AK1}).



Assume that $E(\cM_1,\tau_1)$ and $F(\cM_2,\tau_2)$ are   symmetrically $\Delta$-normed operator spaces affiliated with semifinite von Neumann algebras $(\cM_1,\tau_1)$ and $(\cM_2,\tau_2)$, respectively.
The main purpose of this section is to
 describe the general form of order-preserving operator from $E(\cM_1,\tau_1)$ into  $F(\cM_2,\tau_2)$.
 The case of semifinite von Neumann algebras is more complicated than that of finite von Neumann algebras because of the possible absence of the identity in  the  symmetric space $E(\cM_1,\tau_1)$.
The following is quoted from \cite[Section 5.2]{SV}:``the fact that the unit element $1$ in $\cM$ does not belong to any symmetric space $E(\cM,\tau)$ with order continuous norm when $\tau(1)=\infty$ presents additional technical obstacles''.
In
 this section, we completely resolve the case when $E(\cM,\tau)$ has order continuous $\Delta$-norm,  which is left unanswered in
 \cite{SV}.

It is known that for two von Neumann algebras $\cM_1,\cM_2$ and $p>0$, $p\ne 2$,   $\cM_1$ and $\cM_2$ are Jordan $*$-isomorphic  if and only if $L_p(\cM_1,\tau_1)$ are $L_p(\cM_2,\tau_2)$ are isometrically isomorphic \cite{Sherman, Sherman2005}.
In this section,  we prove  that is,  if
 there exists an order-preserving  isometry from $E(\cM_1,\tau_1)$ into $F(\cM_2,\tau_2)$,  then  there exists a Jordan $*$-homomorphism from  $\cM_1$ into $\cM_2$.

For the sake of convenience, we denote  $$(\cM_1)_{fin}:=\{x\in \cM_1 \mid \tau_1(s(x))<\infty\}$$ and $$(\cM_2)_{fin}:=\{x\in \cM_2 \mid \tau_2(s(x))<\infty\} .$$
We always  assume that
the carrier projections of the symmetrically $\Delta$-normed spaces considered in this section  are ${\bf 1 }$.
By Proposition \ref{prop:car}, it is easy to see that $\cP(E)=\cP_{fin} (\cM_1): =(\cM_1)_{fin} \cap \cP(\cM_1)$ whenever $E(\cM_1,\tau_1)\subset S_0(\cM_1,\tau_1)$ and $c_E ={\bf 1_{\cM_1}}$.

If $\tau_1({\bf 1}_{\cM_1})<\infty$, the isometries from one symmetrically normed space into  another are widely studied.
Since ${\bf 1}_{\cM_1}\in E(\cM_1,\tau_1)$,
one can obtain  an  explicit form of the isometry $T:E(\cM_1,\tau_1)\rightarrow F(
\cM_2,\tau_2)$ directly (see e.g. \cite{SV,CMS,Y}).
For convenience of the reader, we present a proof for $\Delta$-normed case, which extends \cite[Proposition 6]{SV} to the case of general order-preserving operators.
We also rectify an oversight in the proof of \cite[Proposition 6]{SV}, where $T({\bf 1}_{\cM_1})^{-1}$ was asserted to be a $\tau_2$-measurable operator,
which requires additional arguments when $(\cM_2,\tau_2)$ is infinite  while is correct for a finite von Neumann algebra $(\cM_2,\tau_2)$.
 The latter matter though only affects the form of the
 Jordan $*$-homomorphisms and not  the result claimed in that proposition.


\begin{theorem}\label{th:fin}
Let $(\cM_1,\tau_1)$ be a finite von Neumann algebra having a faithful normal finite trace $\tau_1$ and $(\cM_2,\tau_2)$ be an arbitrary semifinite von Neumann algebra.
Assume that $E(\cM_1,\tau_1)$ and $F(\cM_2,\tau_2)$ are symmetrically $\Delta$-normed operator spaces.
If there exists an order-preserving linear operator  $T:E(\cM_1,\tau_1)\stackrel{into}{\longrightarrow} F(\cM_2,\tau_2)$ which is disjointness preserving,
then  there exists a  Jordan $*$-mono\-morphism $J:\mathcal{M}_1\to\mathcal{M}_2$.
Furthermore, $T({\bf 1}_{\cM_1})$ commutes with  $T(x)$ and $ J(x)$ for all  $x\in \cM_1$, and
\begin{align}\label{eq:Tfin}
T({\bf 1}_{\cM_1}) J(x) =T(x), ~x \in \cM_1.
\end{align}
Moreover, if $T$ is normal (i.e.   $\sup T(x_i)=T(x)$ for any increasing net $x_i\in E(\cM_1,\tau_1)^+$ with $\sup x_i= x\in E(\cM_1,\tau_1)$),
%
 then $J$ is a normal Jordan $*$-isomorphism onto a  weakly closed $*$-subalgebra of   $\cM_2$.
\end{theorem}
\begin{proof}
Since $\tau_1({\bf 1}_{\cM_1})<\infty$, it follows from Proposition \ref{prop:car} that ${\bf 1}_{\cM_1} \in E(\cM_1,\tau_1)$.

Let $e\in \cP(\cM_1)$.
The disjointness preserving property of $T$ implies that $$[T(e),T({\bf 1}_{\cM_1})] = [T(e),T(e)]+[T(e),T({\bf 1}_{\cM_1} -e)] =0 .$$
In particular, $T({\bf 1}_{\cM_1})$ commutes with $T(e)$, $e \in \cP(\cM_1)$.
In other words, $T(e)$ commutes with every spectral projection of $T({\bf 1}_{\cM_1})$.

Let    $r:=s(T({\bf 1}_{\cM_1}))$.
Since $T$ is order-preserving, it follows that
\begin{align}\label{TXX1}
0\le T(x)\le \|x\|_\infty T({\bf 1}_{\cM_1}), ~0\le x\in \cM_1.\end{align}
Therefore,
$$0\le ( {\bf 1}_{\cM_2}  -r) T(x)( {\bf 1}_{\cM_2}  -r) \le \|X\|_\infty ( {\bf 1}_{\cM_2} -r ) T({\bf 1}_{\cM_1}) ( {\bf 1}_{\cM_2} -r) =0. $$
In other words, $T(x)( {\bf 1}_{\cM_2} -r)=0$, which implies that $T(x)$ is affiliated with the reduced von Neumann algebra $r \cM_2 r$ of $\cM_2$.
Without loss of generality, we assume that $r=s(T({\bf 1}_{\cM_1}))={\bf 1}_{\cM_2}$.

Fix $0\le x \in \cM_1$, $\|x \|_\infty \le 1$.
Set $x _n:= \sum_{k=1}^{n} \frac{k-1}{n}e ^{x}(\frac{k-1}{n}, \frac k n]$, $n\ge 1$.
In particular, $\|x_n-x\|_\infty \rightarrow_n 0$.
Since $T$ preserves order, it follows that
$$\left\|T(x_n)-T(x) \right\|_F = \left\|T(x_n -x)\right\|_F \le \Big\|\left\|x_n -x \right\|_\infty T({\bf 1}_{\cM_1})\Big\|_F  \stackrel{\eqref{ineq:to0}}{\rightarrow} 0$$
and therefore, by \cite[Lemma 2.4]{HLS2017}, $T(x_n)\rightarrow_n T(x)$ in measure topology.



Since $T({\bf 1}_{\cM_1})$ commutes with $T(e)$, $e \in \cP(\cM_1)$, it immediately follows from the definition of $x _n$ that $T({\bf 1}_{\cM_1})$ commutes with $T(x _n)$.
By the preceding paragraph,
$$[T(x_n), T({\bf 1}_{\cM_1})]\rightarrow [T(x ), T({\bf 1}_{\cM_1})]$$ in measure topology (see e.g. \cite{DP2,FK}).
Hence, $T({\bf 1}_{\cM_1})$ commutes with $T(x)$ for all $x\in \cM_1$.

Let $e_n: =E^{ T({\bf 1}_{\cM_1})}(\frac1n,\infty)$, $n>0$.
In particular, $e_n\rightarrow s(T({\bf 1}_{\cM_1})) ={\bf 1}_{\cM_2}$ as $n\rightarrow \infty$.
Now, we have $[e_n T(x)e_n, e_n T({\bf 1}_{\cM_1})e_n] =0$,  $n>0$, $x\in \cM_1$.
For every $n$, we define $J_n$ by setting
\begin{align}\label{eq:defJn}
J_n(x):= (e_nT({\bf 1}_{\cM_1})e_n)^{-1} e_n T(x) e_n , ~x\in \cM_1.
\end{align}
 Here, $(e_nT({\bf 1}_{\cM_1})e_n)^{-1}$ is taken from  the algebra $e_n\cM_2e_n$.
Note that $  (e_nT({\bf 1}_{\cM_1})e_n)^{-1}$  commutes with $T(x)$, $x\in \cM_1$.
By  \eqref{TXX1}, we obtain  that $0\le e_n T(x) e_n \le \|x\|_\infty e_n T({\bf 1}_{\cM_1})e_n  $.
Hence,  \begin{align}\label{boundede}
0\le (e_n T({\bf 1}_{\cM_1})e_n)^{-1/2} e_n T(x) e_n (e_n T({\bf 1}_{\cM_1})e_n)^{-1/2} \le \|x\|_\infty .
\end{align}
That is, $J_n(x) \le \|x\|_\infty$.
Moreover, since $T({\bf 1}_{\cM_1})$ commutes with $T(x)$, and $e_k$, $k\ge 1$, is the a spectral projection of $T({\bf 1}_{\cM_1})$,
 it follows that for every $m\ge n$,
\begin{align}\label{Jmn}
J_m(x) e_n&=(e_mT({\bf 1}_{\cM_1})e_m)^{-1} e_m T(x) e_m e_n \nonumber \\ &= (e_nT({\bf 1}_{\cM_1})e_n)^{-1} e_n T(x) e_n  =  J_n(x). \end{align}
Hence, $\{J_n(x)\}_n$ converges in the strong operator topology.
 Define \begin{align}
 \label{eq:defJ}
 J(x):= (so)-\lim_n J_n(x) \in \cM_2 .
 \end{align}
  Clearly, $J$ is a complex-linear mapping.
Since $T({\bf 1}_{\cM_1})$ commutes with $J_n(x)$, it follows that $T({\bf 1}_{\cM_1})$ commutes with $J(x)$, $x\in \cM_1$.
Since every  $\|J_n(x)\|_\infty \stackrel{\eqref{boundede}}{\le } \|x\|_\infty$, $x\in \cM_1$, for every $n>0$ and $J(x)$ is the (so)-limit of $\{J_n(x)\}$,   it follows that  $J $ is a bounded mapping with $\|J(x)\|_\infty \le \|x\|_\infty$, $x\in \cM_1$.


For every $0\ne e\in \cP(\cM_1)$, since $T({\bf 1}_{\cM_1}-e)T(e)=0$, it follows that
 $T({\bf 1}_{\cM_1})s(T(e)) =(T({\bf 1}_{\cM_1}-e)+T(e) )s(T(e))  =T(e)$.
Hence,
we have
\begin{align}\label{JE=STEN}
J_n(e) =(e_n T({\bf 1}_{\cM_1})e_n)^{-1} e_n T(e) e_n &= (e_n T({\bf 1}_{\cM_1}) e_n )^{-1}  e_n
T({\bf 1}_{\cM_1})  e_n s(T(e)) \nonumber \\
&  =e_ns(T(e )) .
\end{align}
Taking $n\rightarrow \infty$, we obtain that
\begin{align}\label{JESTE}
 J(e) = s(T(e))>0.
\end{align}
By the disjointness preserving property of $T$, it is easy to see that $J$ is an ortho-homomorphism from $\cP(\cM_1)$ to $\cP(\cM_2)$.
By Lemma \ref{linear orthomorphism extension}, $J$ is  a Jordan $*$-homomorphism.
Moreover,
Proposition \ref{prop:fai} implies that $J$ is a Jordan $*$-monomorphism.

By \eqref{Jmn}, we have $J(x)e_n  =J_n (x), ~x\in \cM_1 $, $n>0$.
Hence,   we have $$T({\bf 1}_{\cM_1})  J(x)e_n =e_nT({\bf 1}_{\cM_1})e_n J(x)e_n =e_nT({\bf 1}_{\cM_1})e_n J_n(x)\stackrel{\eqref{eq:defJn}}{=}   T(x) e_n, ~x\in \cM_1.  $$
Since $e_n\rightarrow {\bf 1}_{\cM_2} $ in the local measure topology (see e.g. \cite[Proposition 2.(v)]{DP2}), it follows that $T({\bf 1}_{\cM_1})  J(x)=T(x)$ (see page \pageref{lmtc}), which proves the validity of \eqref{eq:Tfin}.

Now, assume that $T$ is normal $\{x_\alpha \}$ is an increasing net such that
 $  x_\alpha \uparrow_\alpha  x \in \cM_1^+$.
Since $x_\alpha \uparrow x$, it follows that
\begin{align}\label{eq:Fto0}
T({\bf 1}_{\cM_1} ) J(x) \stackrel{\eqref{eq:Tfin}}{=}  T(x)=\sup _\alpha T(x_\alpha)\stackrel{\eqref{eq:Tfin}}{=} \sup _\alpha  T({\bf 1}_{\cM_1} ) J(x_\alpha )
\end{align}
Since $ T({\bf 1}_{\cM_1} ) $ commutes with $J(x) $ and $J(x_\alpha)$,
it follows from \eqref{eq:Fto0}  
that
\begin{align}\label{ineq:T1TO0}
T({\bf 1}_{\cM_1} )^{\frac12}  J(x)T ({\bf 1}_{\cM_1} )^{\frac12}  =\sup _\alpha T({\bf 1}_{\cM_1} )^{\frac12}J(x_\alpha)T({\bf 1}_{\cM_1} )^{\frac12} .
\end{align}
Since Jordan $*$-homomorphism preserves order (see \eqref{JA>0}),
it follows from  \cite[Proposition 1]{DP2} that
\begin{align*}
T({\bf 1} _{\cM_1})^{1/2}  \sup_\alpha J(x_\alpha)T({\bf 1}_{\cM_1} )^{1/2}& = \sup_\alpha T({\bf 1} _{\cM_1})^{1/2}   J(x_\alpha)T({\bf 1}_{\cM_1} )^{1/2} \\&\stackrel{\eqref{ineq:T1TO0}}{=}  T({\bf 1} _{\cM_1})^{1/2} J(x) T({\bf 1}_{\cM_1} )^{1/2}.
\end{align*}
  Multiplying by $(e_n T({\bf 1}_{\cM_1})e_n)^{-1/2}$ on both sides, we obtain that $e_n  \sup_\alpha J(x_\alpha) e_n =e_n J(x) e_n$.
  Taking $n\rightarrow \infty$, we have $$\sup_\alpha  J(x_\alpha)  =J(x).$$
  That is, $J$ is normal.
The last assertion follows from Remark \ref{isotoJM}.
\end{proof}

\begin{remark}\label{remark:ordtonormal}
If $E(\cM_1,\tau_1)$ has order continuous $\Delta$-norm ($\tau_1$ is possibly infinite), then every    order-preserving  ($\left\|\cdot\right\|_E-\left\|\cdot \right\|_F$)-continuous operator   $T:E(\cM_1,\tau_1)\to F(\cM_2,\tau_2)$ is normal.
Indeed,  assume that  $\{x_\alpha \}$ is an increasing net such that
 $  x_\alpha \uparrow_\alpha  x \in E(\cM_1,\tau_1)^+$.
Since $x_\alpha \uparrow x$, it follows that $ \left\|x -x_\alpha\right\|_E \rightarrow 0$, and therefore,
$$
\left\|T(x) -T(x_\alpha)\right\|_F\rightarrow 0 .
$$
Since $T$ preserves order, it follows from Lemma \ref{cor:cone}
that $T(x_\alpha)\uparrow T(x)$.
\end{remark}

\begin{remark}
Yeadon's proof \cite{Y} (see also \cite[Theorem 3.1]{JRS}) provides an alternative construction of the Jordan $*$-isomorphism which coincides with that given in Theorem \ref{th:fin}.
By the ``disjointness preserving'' property of the $L_p$-isometry, $1\le p<\infty$, $p\ne 2$,  one can construct  a projection  ortho-morphism $\phi$ on $\cP_{fin}(\cM_1)$ by defining $\phi(e):=s(T(e))$, $e\in \cP_{fin}(\cM_1)$.
This  ortho-morphism can be extended to a   Jordan $*$-monomorphism.\end{remark}

Now, we can  use  Lemma \ref{familyofF}   to prove the semifinite case.

\begin{theorem}\label{th:J*}
 Assume that $E(\cM_1,\tau_1)$ and $F(\cM_2,\tau_2)$ are symmetrically $\Delta$-normed operator spaces.
If there exists a normal order-preserving operator $T:E(\cM_1,\tau_1)\stackrel{into}{\longrightarrow} F(\cM_2,\tau_2)$ which is disjointness preserving,
then there exists a normal Jordan $*$-isomorphism $J$ from $\cM_1$ onto a weakly closed $*$-subalgebra of   $\cM_2$ such that    $T(e)J(x)=T(x)$ for every $X\in e\cM_1 e$, $e\in \cP_{fin}(\cM_1)$.
\end{theorem}
\begin{proof}

If $\tau_1({\bf 1}_{\cM_1})<\infty$, then the assertion follows from Theorem \ref{th:fin}.

If $\tau_1({\bf 1}_{\cM_1})=\infty$,
then
we define a normal Jordan $*$-homomorphism $J_e$ on the reduced algebra $e\cM_1 e$, $e\in\cP_{fin}(\cM_1)$,
as in Theorem \ref{th:fin}.
If $e\le f \in \cP_{fin}(\cM_1)$,
then
$ J_f(e)\stackrel{\eqref{JESTE}}{=}s(T(e))$.
Let $0\le x\in e\cM_1 e$.
Set $x _n:= \sum_{k=1}^{n} \frac{k-1}{n}e ^{x}(\frac{k-1}{n}, \frac k n]$, $n\ge 1$.
It follows from the linearity and disjointness-preserving property that  $J_f(x_n) =J_e (x_n)$ for any $n \ge 1$.
By the normality of $J_e$ and $J_f$,
 we have
\begin{align*}J_f(x)=J_e(x) , ~0\le x\in e\cM_1 e,
\end{align*}
which implies
that  $J_f$ coincides with $J_e$  on $e\cM_1 e $.
 By Lemma \ref{familyofF}, we complete the proof.
\end{proof}

\begin{remark}\label{TXJ1E0}

For every $0\le x\in  E(\cM_1,\tau_1)$, we have $s(T(x)) = J(s(x))$.
Indeed,
assume that $\{x_\alpha\in  (\cM_1)_{fin}\}$ be an upwards directed net increasing to $x\in E(\cM,\tau)^+$ (see \cite[Proposition 1]{DP2} or \cite{DPS}).
We have
$$T(x_\alpha )\stackrel{\eqref{eq:Tfin}}{=}T(s(x_\alpha)) J(x_\alpha)\stackrel{\eqref{eq27}}{=} T(s(x_\alpha))J(x_\alpha )J(s(x_\alpha)),$$
which implies that $s(T(x_\alpha))\le   J(s(x_\alpha))\le J(s(x))$.
Note that $T(x_\alpha)\uparrow T(x)$ (by the normality of $T$).
This implies, in particular that, $s(T(x)) \le J(s(x))$.

On the other hand, $T(x) \ge  T(x_\alpha)$, which implies that $$s(T(x)) \ge s(T(x_\alpha)) \ge s(T(\frac 1n e^{x_\alpha}(\frac1n,\infty ))) \stackrel{\eqref{JESTE}}{=} J(e^{x_\alpha}(\frac1n,\infty )) .$$
Hence, taking $n\rightarrow \infty$, by the normality of $J$,
we have $s(T(x)) \ge J(s(x_\alpha )) $.
 Taking the (so)-limit of $J(s(x_\alpha))$, by the normality of $J$, we obtain that  $s(T(x)) = J(s(x))$.
\end{remark}

The techniques used in Yeadon's  proof of \cite[Theorem 2]{Y} rely on the fine properties of $L_p$-norms.
However,  for general symmetrically $\Delta$-normed spaces, we do not have explicit descriptions of the $\Delta$-norms, which was the main obstacle we encountered.
In Theorem \ref{cor:BJX} below, we  use  an approach different from that used in \cite{Y} to describe disjointness preserving order-preserving operators from general symmetrically  $\Delta$-normed space into another,
 which unifies  and extends  the results in \cite{SV,Broise,CMS}.
In particular, we establish a noncommutative version of Abramovich's theorem \cite{A1983}.

\begin{theorem}\label{cor:BJX} Assume that $E(\cM_1,\tau_1)$ and $F(\cM_2,\tau_2)$ are symmetrically $\Delta$-normed operator spaces.
If there exists a normal order-preserving operator $T:E(\cM_1,\tau_1)\stackrel{into}{\longrightarrow} F(\cM_2,\tau_2)$ which is disjointness preserving,
then there exist  a (possibly unbounded) positive self-adjoint $B$ affiliated with $\cM_2$ and a normal Jordan $*$-isomorphism $J$ from $\cM_1$ onto a weakly closed $*$-subalgebra of   $\cM_2$  such that
\begin{align}\label{main_ineq}
T(x) = B J(x), ~x \in E(\cM_1,\tau_1)\cap \cM_1.
\end{align}
\end{theorem}


In particular, if $\left\|\cdot\right\|_E$ is order continuous, then by Remark \ref{remark:ordtonormal} every order-preserving $\Delta$-norm-continuous  $T$ is automatically normal.
The existence of Jordan $*$-isomorphism follows from Theorem  \ref{th:J*} above.
To define the operator $B$ properly and
prove that this Jordan $*$-isomorphism $J$
satisfies \eqref{main_ineq}, we need some preparations.
The proposition below defines $B$ as the strong resolvent limit of the net $\{T(e)\}_{e\in \cP_{fin}(\cM_1)}$.
We refer for the definition of strong resolvent convergence to \cite[p. 284]{RS}.

\begin{remark}In the setting when $\left\|\cdot \right\|_E$ is order continuous, it is clear that the positivity of isometries in  Theorem \ref{cor:BJX} can be removed by using the same techniques used in \cite{Y} (for a complete exposition, see \cite{JC}).
In this case, there exists a partial isometry $W\in \cM_2$ such that \begin{align*}
T(x) = WB J(x), ~x \in E(\cM_1,\tau_1)\cap \cM_1.
\end{align*}

\end{remark}

\begin{proposition}\label{3.9}
There exists a limit $B$ of $\{T(e)\}_{e\in \cP_{fin}(\cM_1)}$ in the strong resolvent sense.
In particular, $B$ is affiliated with $\cM_2$.
\end{proposition}
\begin{proof}

For each   $e\in \cP_{fin}(\cM_1)$, we have the following spectral resolution
$$T(e) = \int_0^\infty \lambda  d P_e(\lambda ) .$$
 In particular, $J(e) \stackrel{\eqref{JESTE}}{=} s(T(e))= {\bf 1}_{\cM_2} - P_e(0)$.
For every projection $f\le e  $, by Theorem \ref{th:fin},
we have $T(f) =T(e) J(f)= J(f) T(e)$ and hence,
for $\lambda \ge 0$, by the spectral theorem, we have
\begin{align}\label{eq:JFTOTF}
{\bf 1}_{\cM_2}  - P_f(\lambda ) =J(f )  ({\bf 1}_{\cM_2} -P_e (\lambda ))=  ({\bf 1}_{\cM_2}  -P_e (\lambda ))J(f) .
\end{align}
It follows from \eqref{eq:JFTOTF} that $\{P_e (\lambda) \}_{E\in \cP_{fin}(\cM_1)}$ is a decreasing net indexed by the upwards-directed set of projections in $\cP_{fin}(\cM_1)$.
Now, by Vigier's theorem (see \cite[Theorem 2.1.1]{LSZ}),
we can define
\begin{align}\label{defPlambda}
P_\lambda :=so- \lim_{e \in \cP_{fin}(\cM_1)} P_e (\lambda), ~\lambda \in \mathbb{R}.
\end{align}
In particular, for every $\lambda <0$, we have $P_\lambda =so- \lim_{e \in \cP_{fin}(\cM_1)} P_e(\lambda) =0$.

To show that  the   limit $B=\lim_{E \in \cP_{{\rm fin}}(\mathcal{M}_1)}T(E)$
exists in strong resolvent sense, i.e.,
$$B:= \int_0^\infty \lambda d P_\lambda$$
is well-defined,
it suffices to show that
$\{P_\lambda\}_{\lambda \in \mathbb{R}}$ is a spectral family.

  It follows immediately from the definition that
$P_\lambda P_\mu =P_\mu P_\lambda =P_\lambda$, $\lambda\le \mu$.

On one hand, $P_\lambda =\inf_{e \in \cP_{fin}(\cM_1)} P_e (\lambda )=\inf_{e \in \cP_{fin}(\cM_1)} \inf_{\varepsilon >0} P_e (\lambda +\varepsilon) \ge \inf_{\varepsilon >0} P_{\lambda +\varepsilon}$.
On the other hand, $P_\lambda =\inf_{e \in \cP_{fin}(\cM_1)} P_e (\lambda )\le \inf_{e \in \cP_{fin}(\cM_1)} P_e(\lambda +\varepsilon)$ for every  $\varepsilon >0$, that is, $P_\lambda  \le P_{\lambda +\varepsilon}$.
It follows that $\lambda  \mapsto P_\lambda$ is (so)-right-continuous.

Since $P_\lambda \stackrel{\eqref{defPlambda}}{=} so-\lim_{e\in \cP_{fin}(\cM_1)}P_e(\lambda) $,  it follows from \eqref{eq:JFTOTF} that
for every $f \in \cP_{fin}(\cM_2)$
 that
\begin{align}\label{eq:PLAMBDA}
{\bf 1}_{\cM_2}  - P_f(\lambda ) =  ({\bf 1}_{\cM_2}  -P_\lambda )J(f ) .
\end{align}
Taking $\lambda\rightarrow +\infty$, we obtain from
 \eqref{eq:PLAMBDA} that $0 = \lim_{\lambda \rightarrow  + \infty} ({\bf 1}_{\cM_2}  -P_\lambda )J(f)$.
 Hence, $\left(\lim_{\lambda\rightarrow  +\infty} ({\bf 1}_{\cM_2}  -P_\lambda ) \right)J({\bf 1}_{\cM_1})  =0$.
Since ${\bf 1}_{\cM_2}  -P_f(\lambda) \le J(f) \le J({\bf 1}_{\cM_1}) $ for every $\lambda >0$,
it follows that ${\bf 1}_{\cM_2}  -P_\lambda \le J({\bf 1}_{\cM_1}) $ for every $\lambda>0$.
Therefore, $\lim_{\lambda\rightarrow  +\infty} ({\bf 1}_{\cM_2}  -P_\lambda ) =0$, i.e., $\lim_{\lambda\rightarrow  +\infty} P_\lambda = {\bf 1}_{\cM_2}$.

Note that $$\lim_{\lambda \rightarrow -\infty}P_\lambda = \inf_{\lambda} P_\lambda =\inf_{\lambda}\inf _{e \in \cP_{fin}(\cM_1)} P_e(\lambda)=0 .$$
Therefore,  $\{P_\lambda\}$ is a spectral family. That is, $B$ is  a well-defined self-adjoint  operator (see e.g. \cite{Helffer}).
Since $\{P_\lambda\}\subset \cM_2$, it follows that $B$ is affiliated with $\cM_2$ (see e.g. \cite[Proposition II 1.4]{DPS}).
\end{proof}

We should prove that $T(x)=B J(x)  $ for every $x \in E(\cM_1,\tau_1)\cap \cM_1$ (see \eqref{main_ineq}).
We first prove the following proposition.
\begin{proposition}\label{propEfinandcom}
Let $B= \int_0^\infty \lambda d P_\lambda$ be defined as in Proposition \ref{3.9}.
Then,
\begin{enumerate}
  \item $P_\lambda$ commutes with $J(x)$ for every $x \in \cM_1$.
  \item for every   $0 \le x \in E(\cM_1,\tau_1)\cap \cM_1\cap S_0(\cM_1,\tau_1)$ (respectively,  $0 \le x \in E(\cM_1,\tau_1)\cap \cM_1$ with $x\notin S_0(\cM_1,\tau_1)$) and spectral projection $e\in \cP_{fin}(\cM_1)$ (respectively, any spectral projection) of  $x$, we have
\begin{align}\label{eq:TXEBJXJE}
T(xe) = BJ(x)J(e) .
\end{align}
\end{enumerate}
\end{proposition}
\begin{proof}
(1). Recall that for every $f\in \cP_{fin}(\cM_1)$, we have $s(T(f)) =J(f)$ (see e.g. Remark \ref{TXJ1E0}).
Hence,
\begin{align}\label{com1-}
({\bf 1}_{\cM_2} - P_f( \lambda) ) J(f)={\bf 1}_{\cM_2} - P_f( \lambda)  \stackrel{\eqref{eq:PLAMBDA}}{=} J(f)  ({\bf 1}_{\cM_2} -P_\lambda )=({\bf 1}_{\cM_2} -P_\lambda ) J(f)
\end{align}
  for every $f \in \cP_{fin}(\cM_1) $, which implies that
\begin{align}\label{commute}
P_f( \lambda) J(f ) =P_\lambda J(f) =J(f) P_\lambda .
\end{align}
Since $P_\lambda J(f) \stackrel{\eqref{commute}}{=} J(f) P_\lambda $
 for every $f \in \cP_{fin}(\cM_1) $,
 by the  normality and linearity  of $J$, we obtain that $P_\lambda$ commutes with $J(x )$ for every $x \in \cM_1$.

(2). Observe that
\begin{align}\label{eq:fin}
e^{T(f)}(  \lambda ,\infty )={\bf 1}_{\cM_2} - P_f( \lambda)  \stackrel{\eqref{com1-}}{=} ({\bf 1}_{\cM_2} -P_\lambda ) J(f)  =e^{B}(  \lambda ,\infty )  J(f)
\end{align}
for any $f\in (\cM_1)_{fin}$.
For any projection $e\in E(\cM_1,\tau_1)$, by the normality of $T$,
$$\sup_{e_\alpha \le e ,~ e_\alpha \in \cP_{fin}(\cM_1)} T(e_\alpha ) =T(e). $$
By Remark \ref{TXJ1E0}, we have
\begin{align*}
T(e) J(e_\alpha) &=T(e-e_\alpha)J(e_\alpha) +T(e_\alpha)J(e_\alpha) \\
&= T(e-e_\alpha) J(e-e_\alpha) J(e_\alpha) +T(e_\alpha)J(e_\alpha) \\
& \stackrel{\eqref{eq27}}{=}T(e_\alpha)  J(e_\alpha) \\
& \stackrel{\eqref{eq:Tfin}}{=}T(e_\alpha).
\end{align*}
Hence, $e^{T(e)}(\lambda,\infty) J(e_\alpha) =e^{T(e_\alpha)}(\lambda,\infty)$.
By Remark \ref{TXJ1E0}, we have $s(T(e)) =J(e)$.
By the normality of $J$, taking the limit of $e_\alpha$, we have
$$ e^{T(e)}(\lambda,\infty)=e^{T(e)}(\lambda,\infty) J(e ) = \sup_\alpha e^{T(e_\alpha)}(\lambda,\infty). $$
In particular, $e^{T(e_\alpha)}(  \lambda ,\infty ) \rightarrow e^{T(e)}(  \lambda ,\infty )$ in strong operator topology  (see \cite[Theorem 2.1.1]{LSZ} for Vigier's theorem).
By the normality of $J$, we have
$$e^{T(e)}(  \lambda ,\infty ) \stackrel{\eqref{eq:fin}}{=} e^B(\lambda,\infty)J(e).$$
Hence,
\begin{align}\label{X_n}
T(e)= BJ(e).
\end{align}

The assertion follows from Proposition \ref{JXJE}.
Indeed, for every   $0 \le x \in E(\cM_1,\tau_1)\cap \cM_1 \cap S_0(\cM_1,\tau)$ and spectral projection $e \in \cP_{fin}(\cM_1) $ of  $x $, we have
\begin{align*}
T(xe) \stackrel{\eqref{eq:Tfin}}{=} T(e)J(xe) \stackrel{\eqref{X_n}}{=}  BJ(e)J(xe)\stackrel{\eqref{eq27}}{=}  BJ(xe) \stackrel{\eqref{eq27}}{=} BJ(x)J(e) .
\end{align*}
For the case when $x$ is not $\tau$-compact, it is clear that every  ($\tau$-finite or infinite) projection is in $E(\cM_1,\tau_1)$.
By the normality of $T$, for any spectral projection $e  $ of  $x $, we have
$$T(xe)=\sup_{f\in \cP_{fin} (\cM_1),~ f\le e} T((xe)^{1/2} f (xe)^{1/2}). $$
Note that $s\left((xe)^{1/2} f (xe)^{1/2} \right)$ is $\tau_1$-finite for any $\tau_1$-finite projection $f$.
By Theorem \ref{th:J*} and Remark \ref{TXJ1E0},  we have
\begin{align*}
T(xe)&=\sup_{f\in\cP_{fin} (\cM_1),~ f\le e }  T\left(s\left((xe)^{1/2} f (xe)^{1/2} \right)\right)  J((xe)^{1/2} f (xe)^{1/2})\\
&=\sup_{f\in \cP_{fin} (\cM_1),~ f\le e}
 T\left(  e \right)  J((xe)^{1/2} f (xe)^{1/2}).
 \end{align*}
By Theorem \ref{th:fin} and the normality of $T$, $T\left(  e \right)$ commutes with $J((xe)^{1/2} f (xe)^{1/2})$.
Hence,  by the normality of $J$, we get
 \begin{align*}
T(xe)= T\left(  e \right)  J(x e).
\end{align*}
Therefore, we have
\begin{align*}
T(xe)=  T(e)J(xe) \stackrel{\eqref{X_n}}{=}  BJ(e)J(xe)\stackrel{\eqref{eq27}}{=}  BJ(xe) \stackrel{\eqref{eq27}}{=} BJ(x)J(e) .
\end{align*}
\end{proof}

Since $J(x)$, $x\in \cM_1$, commutes with $P_\lambda$, $\lambda \in \mathbb{R}$, it follows that $BJ(x)$ is a self-adjoint operator \cite{DNN,Sch}.
Moreover,  since $B$ is positive, it follows  $BJ(x)$ is also positive.
Recall that $B$ is affiliated with $\cM_2$ (see Proposition \ref{3.9}).
It is easy to see that $BJ(x)$ is affiliated with $\cM_2$ (see e.g. \cite[Proposition II 1.4]{DPS}).
To avoid dealing with the domain of unbounded operators, we first show that $BJ(x)$ is $\tau_2$-measurable.

\begin{proposition} \label{3.11} $BJ(x)=T(x)$ for every $0\le x \in \cM_1 \cap E(\cM_1,\tau_1)$.
\end{proposition}
\begin{proof}
By Proposition \ref{propEfinandcom},
it suffices to prove the case when $x\in S_0(\cM_1,\tau_1)$.
Let $e_n =e^{x}(\frac1n,n)$, $n>1$.
 In particular, $e_n$ is $\tau_1$-finite and
 $\sup_n e_n =s(x )$.
Since $BJ(x)$ is positive and self-adjoint,
it follows that $BJ(x)$ has a spectral resolution
$$BJ(x) =\int_0^\infty \lambda d Q(\lambda).$$
Since  $P_\lambda$, $J(x)$ and $J(e_n)$ commute with each other (see Propositions \ref{propEfinandcom} and  \ref{JXJE}), which implies that $J(e_n)$ strongly commutes with $BJ(x)$ (see e.g. \cite[Theorem 1]{DNN}).
 Defining $Q_n(\lambda )$ by  ${\bf 1}_{\cM_2} - Q_n(\lambda ) =({\bf 1}_{\cM_2} -Q(\lambda ))J(e_n)$, we have
$$
T(xe _n)\stackrel{\eqref{eq:TXEBJXJE}}{=}BJ(x)J(e_n) =\int_0^\infty \lambda d Q_n(\lambda ).
$$
By Remark \ref{TXJ1E0}, we know that  $T(xe_n) =T(xe_n)J(e_n) +T(x(1-e_n))J(e_n) =T(x)J(e_n)=J(e_n) T(x)$.
Let $T(x):=\int_0^\infty \lambda d Q_{T(x)}(\lambda )$ be the spectral resolution.
In particular, ${\bf 1}_{\cM_2} - Q_n(\lambda ) =({\bf 1}_{\cM_2} - Q_{T(x)}(\lambda )) J(e_n )$.
Hence, $({\bf 1}_{\cM_2} - Q_{T(x)}(\lambda )) J(e_n ) =({\bf 1}_{\cM_2} -Q(\lambda ))J(e_n).$
Taking the (so)-limit of $J(e_n)$ and using  the normality of $J$, we have $$({\bf 1}_{\cM_2} - Q_{T(x)}(\lambda )) J(s(x) ) =({\bf 1}_{\cM_2} -Q(\lambda ))J(s(x)).$$
By Proposition  \ref{JXJE}, we have
   $J(s(x)) \ge   s(J(x)) \ge s(BJ(x)) \ge {\bf 1}_{\cM_2} -Q(\lambda )$.
   On the other hand, Remark \ref{TXJ1E0} implies that  $J(s(x)) \ge  {\bf 1}_{\cM_2} - Q_{T(x)}(\lambda ) $.
   We conclude that $$ {\bf 1}_{\cM_2} - Q_{T(x )}(\lambda )={\bf 1}_{\cM_2} -Q(\lambda ) ,~ \lambda >0,$$   i.e., $T(x) =BJ(x)$, $0\le x \in \cM_1 \cap E(\cM_1,\tau_1)$.
\end{proof}

\begin{proof}[\textbf{Proof of Theorem \ref{cor:BJX}}]
Now, we consider the general case when $x\in \cM_1\cap E(\cM_1,\tau_1)$ is not necessarily positive.
For any $x\in \cM_1\cap E(\cM_1,\tau_1)$,
let $J(x)^* = u|J(x)^*|$ be polar decomposition.
By  \eqref{lemma:desofabs}, we have that
$$BJ(x) =B|J(x)^*| u^*  =B|J(x^*)| u^*  =  BJ(p|x^*| +({\bf 1}_{\cM_1} -p)|x|)u^* ,$$
where $p$ is a central projection in $\cM_1$ defined as in \eqref{lemma:desofabs}.
Since $p|x^*| +({\bf 1}_{\cM_1} -p)|x|\in \cM_1 \cap E(\cM_1,\tau_1)$, it follows from Proposition \ref{3.11} that $BJ(p|x^*| +({\bf 1}_{\cM_1} -p)|x|)\in S(\cM_2,\tau_2)$.
Hence,  we obtain that $BJ(x)\in S(\cM_2,\tau_2)$.


For every $x\in \cM_1\cap E(\cM_1,\tau_1)$, let  $x_{1+}, x_{1-},x_{2+}, x_{2-}\in \cM_1 $ be positive operators such that $x=(x_{1+} - x_{1-}) + i (x_{2+} - x_{2-}) $.
Since $P_\lambda$,  $\lambda >0$, commutes with $J(x)$ (see Proposition \ref{propEfinandcom}),  we obtain that
\begin{align*}
 T(x)  P_\lambda  &= T( (x_{1+} - x_{1-}) + i (x_{2+} - x_{2-}) ) P_\lambda \\
 &=T( x_{1+} ) P_\lambda + T( - x_{1-}) P_\lambda+T(  i x_{2+} ) P_\lambda+T( - i  x_{2-} ) P_\lambda\\
  & = B J( x_{1+} ) P_\lambda -  BJ(  x_{1-}) P_\lambda+ i BJ(  x_{2+} ) P_\lambda - i BJ( x_{2-} ) P_\lambda\\
    & = B  P_\lambda J( x_{1+} ) -  B P_\lambda J(  x_{1-}) + i B  P_\lambda J(  x_{2+} )  - i B P_\lambda J( x_{2-} )  \\
  &= B P_\lambda  J(x) = B  J(x)P_\lambda .
 \end{align*}
Hence, we obtain that 
$$ (T(x) - B  J(x)) P_\lambda =0. $$
Since  $T(x),BJ(x)\in S(\cM_2,\tau_2)$ and $P_\lambda \uparrow {\bf 1}_{\cM_2}$ as $\lambda \rightarrow \infty$,
 it follows that
$T(x)=BJ(x)$ (see \cite[Proposition 2 and Section 2.5]{DP2}).
\end{proof}

The following is a noncommutative version of  Huijsmans-Wickstead theorem \cite{Huijsmans_W} (see also the Huijsmans-de Pagter-Koldunov theorem  \cite[Theorem 2.2]{AK1}).
\begin{corollary}  Assume that $E(\cM_1,\tau_1)$ and $F(\cM_2,\tau_2)$ are symmetrically $\Delta$-normed operator spaces.
Let $T:E(\cM_1,\tau_1) \longrightarrow  F(\cM_2,\tau_2)$ be a normal order-preserving  injective  operator which is disjointness-preserving,
then $T$ is a d-isomorphism, i.e., $T^{-1}$ is disjointness-preserving from the range of $T$ onto $E(\cM_1,\tau_1)$.
\end{corollary}
\begin{proof}By
Theorem \ref{cor:BJX},
 there exist  a (possibly unbounded) positive self-adjoint $B$ affiliated with $\cM_2$ and a normal Jordan $*$-isomorphism $J$ from $\cM_1$ onto a weakly closed $*$-subalgebra of   $\cM_2$  such that
$$
T(x) = B J(x), ~x \in E(\cM_1,\tau_1)\cap \cM_1.
$$
Assume that $x,y\in E(\cM_1,\tau_1)$ such that  $T(x),T(y)\ge 0$ with $T(x)T(y)=0$.
By the disjointness-preserving property and order-preserving property, we obtain that $x,y\ge 0$.
For $0\le x,y\in E(\cM_1,\tau_1)\cap \cM_1$, we can find nets $\{x_\alpha\}$ and $\{y_\beta\}$ in $(\cM_1)_{fin}^+$  such that $x_\alpha\uparrow x$ and $y_\beta \uparrow y$ (see \cite[Proposition 1]{DP2}).
It suffices to show that if $T(x)T(y)=0$, then $x_\alpha y_\beta =0$ for any $\alpha$ and $\beta$.
Recall that
$$so-\lim_n(e_n T(s(x_\alpha)) e_n)^{-1}T(x_\alpha) \stackrel{\eqref{eq:defJ}}{=}  J(x_\alpha) $$
and
$$so-\lim_n (f_n T(s(y_\beta))f_n )^{-1}T(y_\beta)  \stackrel{\eqref{eq:defJ}}{=}    J(y_\beta),$$
where $e_n =e^{T(s(x_\alpha))}(\frac1n,\infty)$ and $f_n =e^{T(s(y_\beta))}(\frac1n,\infty)$.
Since $T(x)  T(y)=0$, it follows that $T(x_\alpha) T(y_\beta) =0$.
Recall that  $T(s(x_\alpha))$ commutates with $J(x_\alpha)$ and $T(x_\alpha)$, and $T(s(y_\beta))$ commutates with $J(y_\beta)$ and $T(y_\beta)$ (see Theorem \ref{th:fin}).
This implies that $$J(x_\alpha)J(y_\beta)=0 . $$

Note that $J^{-1}$ is a Jordan $*$-isomorphism from the weakly closed $*$-subalgebra of $\cM_2$ onto $\cM_1$ (see e.g. \cite[Appendix A]{RR}).
Hence,
it follows from Proposition \ref{JXJE} that $x_\alpha   y_\beta =0$.
Taking the limit in  local measure topology (see \cite[Proposition 2]{DP2}), we get $xy=0$.
\end{proof}
\section{Order-preserving isometries into $\Delta$-normed  spaces}\label{SLKM}

 In this section, based on detailed study of logarithmic submajorisation, we establish the disjointness-preserving property of order-preserving isometries on noncommutative symmetrically $\Delta$-normed spaces,  which is the key to describe isometries.
We extend    \cite[Proposition 6]{SV} in two directions.
Firstly, we can consider general semifinite von Neumann algebras instead of finite von Neumann algebras,
resolving the case left in \cite{SV}.
Secondly, we extend significantly the class of symmetrically $\Delta$-normed spaces (even in the normed case) to which the theorem is applicable.
In particular, we can consider the  usual $L_1$-norm, which is not strictly $K$-monotone.

Recall that for a finite von Neumann algebra $\cN$ with a faithful normal finite trace $\tau$, the Fuglede-Kadison determinant was introduced in \cite{FugK}.
In \cite{HaaSch} (see also \cite{DSZ2017,DSZ}), Haagerup and Schultz defined  the Fuglede-Kadison determinant  $det _\cN (x ) \ge 0$ of $x \in S(\cN,\tau) $ such that $ \log_+ \mu(x)\in L_1(0,\infty)$  by the integral:
$$ \log det _\cN (x) =\int _0^{\tau ({\bf 1}_\cN)} \log \mu(t;x) dt.$$
\begin{lemma}\label{lemma:det}
Assume that $\cN $ is a finite von Neumann algebra with a faithful normal finite trace $\tau$.
If $0 \le a \in \cL_{\log}(\cN,\tau)$ and $b \in \cL_{\log}(\cN,\tau)$ is self-adjoint with  $-a\le b \le a$, then $det_\cN (b)\le det_\cN  (a)$.
\end{lemma}
\begin{proof}

For every $\varepsilon>0$, we define $a_\varepsilon := a+\varepsilon {\bf 1}$.
In particular, $a_\varepsilon$ is invertible and $a_\varepsilon^{-1} = \int \frac{1}{\lambda }   de_\lambda^{a_\varepsilon }$.
In particular,  $a_\varepsilon^{-1}$ is  bounded. 
By \cite[Proposition 2.5]{HaaSch} (see also \cite{DSZ2015}), we obtain that $ a_\varepsilon^{-1/2}  b a_\varepsilon^{-1/2}  \in \cL_{\log}(\cN,\tau )$.
Moreover,
$$ -{\bf 1}= -a_\varepsilon^{-1/2} a_\varepsilon  a_\varepsilon^{-1/2} \le  a_\varepsilon^{-1/2}  b a_\varepsilon^{-1/2} \le a_\varepsilon^{-1/2} a_\varepsilon  a_\varepsilon^{-1/2} =  {\bf 1} .$$
Hence,  $\mu(a_\varepsilon^{-1/2} ba_\varepsilon^{-1/2} ) \le 1$. That is, $det_\cN ( a_\varepsilon^{-1/2} ba _\varepsilon^{-1/2} ) \le 1$.
By \cite[Proposition 2.5]{HaaSch}, we have
$$ \frac{det_\cN (b)}{det_\cN (a_\varepsilon) } = \frac{det_\cN (b)}{det_\cN (a_\varepsilon^{1/2})det _\cN (a_\varepsilon^{1/2})} = det_\cN ( a_\varepsilon^{-1/2} b a_\varepsilon^{-1/2} )  \le 1 $$
 Since $\varepsilon$ is arbitrary,  it follows  that  $det_\cN b \le det_\cN  a $.
\end{proof}



This is an extension of the result in \cite{CS1} (see also \cite{DDSZ}).
\begin{lemma}\label{lemma:1}
Assume that
 $\cM$ is a semifinite von Neumann algebra equipped with a semifinite faithful normal trace $\tau$.
Let $a,b\in \cM^\Delta$ .
If $a\ge 0$ and  $b$ is self-adjoint with  $-a \le b \le a$, then $b\prec \prec _{\log} a$.
\end{lemma}
\begin{proof}

Without loss of generality, we may assume that $\cM$ is atomless (see e.g. \cite[Lemma 2.3.18]{LSZ}).
For every $t>0$, we can choose a $p\in \cP(\cM)$ such that
$\tau(p ) =t $ and $\mu(s;b)=\mu(s; pbp)$, $s\in (0,t)$ (see e.g. \cite[Page 953]{DDP1992} or \cite{OV1}).
Note that $-pap\le pbp \le pap $.
By  Lemma \ref{lemma:det}, we obtain that
$$ \int_0^{t} \log \mu(s; pbp)ds=   \log det_{p\cM p}(pbp) \le \log det_{p\cM p} (pap) = \int_0^{t} \log \mu(s; pap)ds.$$
Hence,  we obtain that
$$  \int_0^{t} \log \mu(s;b)ds= \int_0^{t} \log \mu(s;pbp)ds \le  \int_0^{t} \log \mu(s; pap)ds \le \int_0^{t} \log \mu(s; a)ds. $$
Since $t$ is arbitrary, it follows that  $\mu(b)\prec\prec_{\log}\mu( a)$.
\end{proof}

Our next lemma is folklore.
We provide a short proof for the sake of convenience.

\begin{lemma}\label{a=b lemma}
Let $\cM$ be von Neumann algebra with a faithful normal finite trace $\tau$
 and let $0\leq b\leq a\in S(\cM,\tau)$.
  If $\mu(b)=\mu(a)$, then $a=b$.
\end{lemma}
\begin{proof}
Note that  ${\bf 1}\leq {\bf 1}+b\leq {\bf 1} +a$.
Taking inverses, we obtain
$${\bf 1} \geq ({\bf 1}+b)^{-1}\geq ({\bf 1}+a)^{-1}.$$
Subtracting ${\bf 1}$,
 we obtain
\begin{align}\label{ineq:0BA}
0\leq \frac{b}{{\bf 1}+b}\leq\frac{a}{1 +a}.
\end{align}
Since the mapping $t\to\frac{t}{1 +t}$ is increasing, it follows  from \cite[Corollary 2.3.17]{LSZ} that
$$\mu(\frac{b}{{\bf 1}+b})=\frac{\mu(b)}{1+\mu(b)}=\frac{\mu(a)}{1+\mu(a)}=\mu(\frac{a}{{\bf 1}+a}).$$
Since $\tau$ is finite, it follows that
\begin{align}\label{tauBA}
\tau(\frac{b}{{\bf 1}+b})=\tau(\frac{a}{{\bf 1} +a})<\infty .
\end{align}

Letting
$x:=\frac{a}{{\bf 1}+a}-\frac{b}{{\bf 1} +b} \stackrel{\eqref{ineq:0BA}}{\ge} 0  ,$
we have
$x \geq0$ and $\tau(x)\stackrel{\eqref{tauBA}}{=}0$.
The faithfulness of $\tau$ implies that $x=0$.
That is,
$\frac{b}{{\bf 1}+b}=\frac{a}{{\bf 1}+a}.$
Subtracting ${\bf 1}$,  we obtain
$$({\bf 1}+b)^{-1}=({\bf 1}+a)^{-1},$$
which implies that $a=b$.
\end{proof}

The following  lemma was known before for the special case $z\in(L_1+L_{\infty})(\mathcal{M})$ (see  \cite[Lemma 4.4]{CS1994} for a similar result).

\begin{lemma}\label{p=r lemma} Let $(\mathcal{M},\tau)$ be a semifinite von Neumann algebra and let $0\leq z\in S(\cM,\tau)$.
Let $r:=e^{z}{(\lambda,\infty)}$, $\lambda >0$,  and
let $p\in \cP(\mathcal{M})$ be such that $t:=\tau(p)=\tau(r)<\infty$.
 If $\mu(pzp)=\mu(z)$ on $(0,t)$, then $p=r$.
\end{lemma}
\begin{proof} Let $z_1=\max\{z,\lambda\}$.
 For any $s\in (0,t)$,  we have
$$\mu(s;pz_1p)\geq\mu(s;pzp)=\mu(s;z),\quad \mu(s; pz_1p)\leq\mu(s;z_1)=\mu(s;z),$$
where  the last equality follows immediately from the definition of singular value functions (see also \cite[Chapter III, Proposition 2.10]{DPS}).
Therefore,
\begin{align}\label{eq:PZ1P}
\mu(s;pz_1p)=\mu(s; z_1), ~s\in (0,t).
\end{align}

Setting $z_2:=(z-\lambda)_+$,
by \cite[Corollary 2.3.16]{LSZ} and the definition of $z_1$, we have
$$\mu(pz_1p)=\mu(pz_2p+\lambda p)\leq\mu(pz_2p)+\lambda.$$
On the other hand, by \cite[Corollary 2.3.17]{LSZ} and definitions of $z_1$ and $z_2$, we have
$$\lambda+\mu(z_2)=\mu(z_1)\stackrel{\eqref{eq:PZ1P}}{=} \mu(pz_1p)$$
 on $(0,t)$.
Hence, we obtain that
$$\lambda+\mu(z_2)  \le \lambda+\mu(pz_2p)$$
on $(0,t)$.
Since $\mu(pz_2p)\le \mu(z_2)$, it follows that
$$\mu(pz_2p)=\mu(z_2)$$
on $(0,t).$
Since both functions vanish outside of $(0,t)$ (see e.g. \cite[Chapter III, Proposition 2.10]{DPS}),
it follows that
\begin{align}\label{PZ2PZ2}
\mu(pz_2p)=\mu(z_2)
\end{align}
on $(0,\infty)$.
Note that $z_2=rz_2r$.
Let $pr=u|pr|$ be the polar decomposition.
Hence, we obtain  that
$$\mu(pz_2 p)=\mu(pr\cdot z_2\cdot rp )=\mu(u|pr|\cdot z_2\cdot|pr|u^*)=\mu(|pr|\cdot z_2\cdot|pr|).$$
For all $0\le a,b\in S(\cM,\tau)$,
 we have (see e.g. \cite[Lemma 2.3.12 and Corollary 2.3.17]{LSZ})
$$\mu(ab^2a)=\mu(|ba|^2)=\mu^2(|ba|)=\mu^2(ba)=\mu^2((ba)^*)$$
$$=\mu^2(ab)=\mu^2(|ab|)=\mu(|ab|^2)=\mu(ba^2b).
$$
Hence,
\begin{align}\label{PZ2PPR}
\mu(pz_2p)=\mu(z_2^{\frac12}|pr|^2z_2^{\frac12}).
\end{align}

Consider the ($\tau$-finite) reduced  von Neumann algebra $r\mathcal{M}r$.
Recall that $z_2=rz_2r$.
We have
$$x:=z_2,\quad y:=z_2^{\frac12}|pr|^2z_2^{\frac12}\in r\cM r.$$
Clearly, $0\leq y\leq x$ and $\mu(x)\stackrel{\eqref{PZ2PZ2}}{=} \mu(pz_2p) \stackrel{\eqref{PZ2PPR}}{=}\mu(y)$.
 By Lemma \ref{a=b lemma}, we obtain that  $y=x$.
Therefore, we have
$$z_2^{\frac12}({\bf 1}-|pr|^2)z_2^{\frac12}=0.$$
Since $s(z_2)=r$,
 it follows that
\begin{align}\label{R=RPR}
r({\bf 1}-|pr|^2)r=0\Longrightarrow r({\bf 1}-rpr)r=0\Longrightarrow r=rpr.
\end{align}
Note that
$$p=rpr+({\bf 1}-r)p({\bf 1}-r)+rp({\bf 1}-r)+({\bf 1}-r)pr.$$
Since $p$ and $r$ are $\tau$-finite, it follows that
$$\tau(rp({\bf 1}-r))=\tau(({\bf 1}-r)\cdot rp)=0,\quad \tau(({\bf 1}-r)pr)=\tau(r\cdot ({\bf 1}-r)p)=0.$$
Hence,
$$\tau(p)=\tau(rpr)+\tau(({\bf 1}-r)p({\bf 1}-r)).$$
By assumption  that $\tau(p)=\tau(r)$ and by $r\stackrel{\eqref{R=RPR}}{=}rpr$, we have
$$\tau(p)=\tau(r)=\tau(rpr).$$
Thus,
$$\tau(({\bf 1}-r)p({\bf 1}-r))=0.$$
Since $\tau$ is faithful, it follows that
$$({\bf 1}-r)p({\bf 1}-r)=0\Longrightarrow p({\bf 1}-r)=0\Longrightarrow p=pr\Longrightarrow p\leq r.$$
Since $\tau(p)=\tau(r)$,  it follows that $p=r$.
\end{proof}

The following result is well-known in the setting of $\cF(\tau)$.
We extend it to the case of the algebra $S_0(\cM,\tau)$ of $\tau$-compact operators.
\begin{proposition}\label{propXY-YX}
Let $ 0\le x,y\in S_0(\cM,\tau) $.
If $xy=-yx$, then $xy=0$.
\end{proposition}
\begin{proof}
Letting  $p_n:= e^{x}[\frac1n, n]$, we have
 \begin{align}\label{pnx}
 p_nxp_n yp_n =p_nxyp_n =- p_n yxp_n=- p_n yp _n xp_n.
 \end{align}
For the sake of convenience, we define $x_n:= p_nxp_n \ge 0$ and $y_n :=p_n yp_n \ge 0$.
Hence, we have $x_n y_n = - y_nx_n $.
Clearly,  $ x_n\in \cL_1(\cM,\tau)\cap \cM $.
Let $q_{m,n}:= e^{y_n}[\frac1m,m]$.
We have
\begin{align*}
q_{m,n} x_nq_{m,n} q_{m,n}y_n q_{m,n}=q_{m,n}x_n y_nq_{m,n} = - q_{m,n}y_nx_nq_{m,n}\\= -q_{m,n} y_n q_{m,n}q_{m,n} x_nq_{m,n} , ~\forall m,n\ge 1.
\end{align*}
Denote $z^1_{m,n}:= q_{m,n} x_n q_{m,n}$ and $z^2_{m,n}:=q_{m,n} y_n q_{m,n}$.
Note that $0\le z^1_{m,n}, z^2_{m,n}\in \cL_1(\cM,\tau)\cap \cM$ and $z^1_{m,n}z^2_{m,n}=-z^2_{m,n}z^1_{m,n}$.
Hence, $$\tau(z^2_{m,n}z^1_{m,n})= \tau(z^1_{m,n}z^2_{m,n} )=\tau(-z^2_{m,n}z^1_{m,n} ),$$ i.e.,
$\tau((z^1_{m,n})^{1/2} z^2_{m,n} (z^1_{m,n})^{1/2})=\tau(z^2_{m,n}z^1_{m,n})=0$.
The faithfulness of $\tau$ implies that $(z^1_{m,n})^{1/2} z^2_{m,n} (z^1_{m,n})^{1/2}=0$.
Hence, we obtain that $(z^1_{m,n})^{1/2}  (z^2_{m,n} )^{1/2}  =0$, and therefore,
$$q_{m,n} x_n  y_n  q_{m,n} =z^1_{m,n}  z^2_{m,n} =0. $$
Passing  $m\rightarrow \infty$, we obtain that $0=q_{m,n}  x_n  y_n  q_{m,n} \rightarrow_m  s(y_n) x_n y_n$ in the  measure topology (see e.g. \cite[Proposition 2]{DP2} or \cite[Chapter II, Proposition 6.4]{DPS}), i.e., $s(y_n) x_n y_n=0$.
Since $y_n x_n y_n= y_n s(y_n) x_n y_n =0$,
it follows that $x_n^{1/2}y_n=0$ and therefore,
$x_ny_n =x_n^{1/2} x_n^{1/2} y_n =0$.
That is,
 $ p_nxyp _n \stackrel{\eqref{pnx}}{=}0$ for every $n$.
Taking $n\rightarrow \infty$,
we obtain that $0=p_nxyp_n \rightarrow_n  xys(x) $ in measure topology,
which implies that $xyx=xys(x)x=0$.
Hence, $xy^{1/2}=0$, and therefore,
$xy=0$.
This completes the proof.
\end{proof}

The following lemma is an extension of \cite[Theorem 2]{SV}.
\begin{lemma}\label{lemma:XY+-} Let $0\leq x,y\in \cM^\Delta:=\left(\cL_{\log }(\cM,\tau)+\cL_\infty(\cM,\tau)\right)\cap S_0(\cM,\tau)$. If $\mu(x-y)=\mu(x+y)$,
 then $xy=0.$
\end{lemma}
\begin{proof} Let
$$p_{\lambda}=e^{|x-y|}(\lambda,\infty),\quad r_{\lambda}=e^{x+y} (\lambda,\infty),\quad  \lambda >0.$$
Since $\mu(x-y)=\mu(x+y)$, it follows that  $t_{\lambda}: =\tau(r_{\lambda})=\tau(p_{\lambda})<\infty$ (see e.g. \cite[Chapter 3.2]{DPS}).

By definition, $p_{\lambda}$ commutes with $x-y$.
Thus,
$$p_{\lambda}|x-y|p_{\lambda}=|p_{\lambda}(x-y)p_{\lambda}|.$$
On $(0,t_{\lambda})$,
 we have the coincidence of the following functions (see e.g.  \cite[Chapter III, Proposition 2.10]{DPS})
\begin{align}\label{Plambda}
\mu(x+y)=\mu(x-y)=\mu(p_{\lambda}|x-y|p_{\lambda})=\mu(p_{\lambda}(x-y)p_{\lambda}).
\end{align}
For positive operators $a,b\in \cM^\Delta$,  Lemma \ref{lemma:1} implies that
$a-b\prec\prec_{{\rm log}}a+b$.
Therefore,
$$\mu(x+y) \chi_{(0,t_{\lambda})}\stackrel{\eqref{Plambda}}{=}\mu(p_\lambda (x-y) p_\lambda )\prec\prec_{{\rm log}}\mu(p_{\lambda}(x+y)p_{\lambda})   \leq\mu(x+y) \chi_{(0,t_{\lambda})} . $$
 Thus,
$$\mu(p_{\lambda}(x+y)p_{\lambda})=\mu(x+y)$$
on $(0,t_{\lambda})$.
By Lemma \ref{p=r lemma}, we have $p_{\lambda}=r_{\lambda}$.

Since $p_{\lambda}=r_{\lambda}$ for all $\lambda>0$, it follows from the Spectral Theorem that $|x-y|=x+y$.
Squaring both parts of the preceding equality, we arrive at   $xy=-yx$.
Now, we can apply Proposition \ref{propXY-YX} to conclude that $xy=0$.
\end{proof}

The following example shows that one cannot expect for a similar result of Lemma \ref{lemma:XY+-} without the assumption of $\tau$-compactness.
\begin{example}
Let $\cM$ be a semifinite infinite von Neumann algebra with a semifinite faithful normal trace $\tau$.
Consider algebra $\cM \oplus \cM \oplus \cM$ equipped with trace $\tau\oplus \tau\oplus \tau$.
Let $ a:= {\bf 1} \oplus 0 \oplus \frac13$ and $ b:= 0 \oplus {\bf 1} \oplus \frac13$.
It is clear that $\mu(a-b)=\mu(a+b)$, but $ab\ne 0$.
\end{example}

In what follows, we always assume that $E(\cM_1,\tau_1)$ and $F(\cM_2,\tau_2)$ are  symmetrically $\Delta$-normed spaces affiliated with semifinite von Neumann algebras $(\cM_1,\tau_1)$ and $(\cM_2,\tau_2)$, respectively.

\begin{proposition}\label{prop:positive}
Let $T : E(\cM_1,\tau_1) \stackrel{into}{\longrightarrow}F(\cM_2,\tau_2)$ be an order-preserving isometry,
where  $F(\cM_2,\tau_2)$ has   SLM symmetric  $\Delta$-norm.
If $0\le x,y\in E(\cM_1,\tau_1) $ such that  $T(x),T(y)\in \cM_2^\Delta$ and   $xy=0$,
then $T(x)T(y)=0$.
\end{proposition}
\begin{proof}
Since $xy=0$ implies that $|x-y| =|x+y|$, it follows that  $\left\|x-y\right\|_E =\left\|x+y\right\|_E$, and therefore,  $\left\|T(x) -T(y) \right\|_F =\left\|T(x) + T(y) \right\|_F  $.
It follows from Lemma \ref{lemma:1} that   $$T(x)-T(y) \prec\prec_{\log} T(x)+T(y) .$$
By the Definition of SLM $\Delta$-norms, we obtain that  $\mu(T(x)-T(y)) =\mu(T(x)+T(y))$.
It  follows from  Lemma \ref{lemma:XY+-} that  $T(x) T(y)=0$.
\end{proof}

Every strongly symmetric   space of $
\tau$-compact operators, whose central carrier projection is the identity,  is a subspace of $L_0(\cM,\tau):=\left(L_1(\cM,\tau)+\cM \right)\cap S_0(\cM,\tau)$~\cite{DP2} and therefore a subspace of $\cM^\Delta$.
For every symmetric space $E(0,\infty)$ of functions vanishing at infinity, the corresponding operator space  $E(\cM,\tau)$ is a subspace  of $ L_0(\cM,\tau)$ (see e.g. \cite{LSZ, DPS,DDP2,KPS}).
Note  that if a symmetric norm $\left\|\cdot\right\|_F$ is \emph{strictly} monotone with respect to the submajorisation, then it is
\emph{strictly} monotone with respect to the logarithmic submajorisation (see Proposition \ref{Prop:lefunction2} or  \cite{DDSZ}).
However, the inverse is not the case.
For example, the $L_1$-norm is an SLM norm which fails to be strictly $K$-monotone.
As an application of Theorem~\ref{cor:BJX} (together with Remark \ref{remark:ordtonormal}) and Proposition \ref{prop:positive},  we obtain the following result, which  significantly extends  \cite[Proposition 6]{SV}.
\begin{corollary}\label{cor:4.5}
Assume that $E(\cM_1,\tau_1)$ has order continuous  symmetric  $\Delta$-norm   and assume that  $F(\cM_2,\tau_2)\subset   \cM_2^\Delta$ has   SLM symmetric $\Delta$-norm.
For every order-preserving isometry $T:E(\cM_1,\tau_1)\stackrel{into}{\longrightarrow} F(\cM_2,\tau_2)$, there exists a positive operator $B$ and a Jordan $*$-isomorphism $J$ from $\cM_1$ onto a weakly closed $*$-subalgebra of $\cM_2$ such that
$T(x) =B J(x), ~\forall x\in \cM_1\cap E (\cM_1,\tau_1)$.
\end{corollary}

In the special case when $\cM_1$ and $\cM_2$ are finite,
we obtain the following corollary of
Theorem \ref{th:fin} (and Remark \ref{remark:ordtonormal}), which recovers and  extends    \cite[Theorem 1]{Katavolos2} and \cite[Theorem 2]{Russo}.
\begin{corollary}
Let $(\cM_1,\tau_1)$  and
$(\cM_2,\tau_2)$  be two  finite von Neumann algebras with $\tau_1({\bf 1}_{\cM_1}),\tau_2({\bf 1}_{\cM_2}) <\infty$.
Assume that $E(\cM_1,\tau_1)$  is a   symmetrically   $\Delta$-normed space    and assume that  $F(\cM_2,\tau_2)\subset   \cM_2^\Delta$ has   SLM symmetric $\Delta$-norm.
If  an   order-preserving   isometry   $T:E(\cM_1,\tau_1) \stackrel{into}{\longrightarrow} F(\cM_2,\tau_2)$
satisfies that  $T({\bf 1}_{\cM_1}) ={\bf 1}_{\cM_2}$, then $T|_{\cM_1}$  is a
  Jordan $*$-homomorphism from $\cM_1$ into $\cM_2$.
Moreover, if $\left\|\cdot \right\|_E$ is order continuous, then   $T|_{\cM_1}$  is a
  Jordan $*$-isomorphism from $\cM_1$ onto a weakly closed $*$-subalgebra of $\cM_2$.
\end{corollary}

\section{Order-preserving isometries onto $\Delta$-normed spaces}

Order-preserving isometries from a noncommutative $L_2$-space onto another  were  studied in  \cite[Theorem 1]{Broise}.
The  description of order-preserving isometries onto a fully symmetric space affiliated with a finite von Neumann algebra  is given in  \cite[Theorem 3.1]{CMS}.
In this section,
we consider the general form of  surjective isometries, which significantly extends \cite[Theorem 1]{Broise} and   \cite[Theorem 3.1]{CMS}.
We note that in the ``onto'' case, the   order continuity imposed on the $\Delta$-norms in Corollary \ref{cor:4.5} can be dispensed with.

Recall that $\cM^\Delta:= (\cL_{\log}(\cM,\tau)+\cM)\cap S_0(\cM,\tau)$.
Assume that $(\cM_1,\tau_1)$ and $(\cM_2,\tau_2)$ are semifinite von Neumann algebras.
The following lemma has been obtained by Ju.A. Abramovich for order-preserving isometries of arbitrary normed lattices \cite{A1} (see also \cite[Lemma 3.2]{CMS}).
We extend this  result to surjective   isometries on symmetrically $\Delta$-normed spaces.
\begin{lemma} \label{lemma:order}
Assume that $E(\cM_1,\tau_1)$ and $F(\cM_2,\tau_2)$ are  symmetrically $\Delta$-normed spaces.
In addition, we assume that   $\left\|\cdot \right\|_F$ is (not necessarily strictly) log-monotone.
Let $T:E(\cM_1,\tau_1)\stackrel{onto}{\longrightarrow} F(\cM_2,\tau_2)$ be an order-preserving isometry.
If $T(x) \ge 0$, then $x\ge 0$.
\end{lemma}
\begin{proof}
For every self-adjoint $a$, $T(a) =T(a_+ -a_-) = T(a_+) -T(a_-)$, which implies that $T(a)$ is self-adjoint.
Since $T(\real (x)) + i T (\imm (x))=T(x) >0$,
it follows that $\imm (x) =0$,  that is $x=x^*$.

Let $x_+$ and $x_-$ be the positive part and negative part of $x$, respectively.
If $x_+=0$, then $0\ge T(-x_-) =T(x) \ge 0$ implies that $x=0$.
Hence, it suffices to consider the case when $x_+\ne 0$ and  prove that $x_-=0$.

Let $b_1:= T(x_+)$ and $b_2 :=T(x_-)$.
In particular, since $T$ is an order-preserving isometry, it follows that $b_1,b_2 \ge 0$.
Moreover, $b:=b_1-b_2 =T(x) \ge 0$ and $b_1+b_2 =T(x_+ +x_-)=T(|x|)$.

Note that
\begin{align}\label{xk+-1}
\|\alpha(b_1+b_2)\|_F = \|T(\alpha|x|)\|_F = \|\alpha|x|\|_E = \|\alpha x\|_E
\end{align}
for every $\alpha\in \mathbb{C}$.
We assert that
\begin{align}\label{xk+-}
\left\|\alpha x_+ + \alpha k x_- \right\| _E  = \left\|  \alpha b_1 +  \alpha k  b_2\right\|_F  \le \left\|\alpha x\right \|_E
\end{align}
for all $k=1,2,\cdots$ and $\alpha\in \mathbb{C}$.
It follows from \eqref{xk+-1} that \eqref{xk+-} holds for $k=1$.
Assume that it holds for $k=n$.
Noting that $b, b_2 \ge 0$, we obtain that
$$ -( b + n b_2) \le   b - n b_2
\le   b + n b_2. $$
By Lemma \ref{lemma:1}, the logarithmic  monotonicity of $\|\cdot\|_F$ guarantees that
\begin{align}\label{xk6666}
\left\|\alpha   b- \alpha n  b_2 \right\|_F \le \left\| \alpha   b+ \alpha n  b_2 \right\|_F
\end{align}
for every $\alpha\in \mathbb{C}$.
Using the inequality $ 0\le    b+  n b_2 =   b_1 +(n-1) b_2\le   b_1 + n b_2$ and the assumption of induction, we get
\begin{align}\label{xk77}
\left\| \alpha  b_1 - \alpha(n+1) b_2\right\|_F &= \left\| \alpha  b-  \alpha \cdot n   b_2\right\|_F \stackrel{\eqref{xk6666}}{\le} \left\|  \alpha b+ \alpha\cdot n b_2 \right\|_F \nonumber\\
& \le \left\| \alpha   b_1 + \alpha\cdot n  b_2\right\|_F\stackrel{\eqref{xk+-}}{\le}  \left\|\alpha   x\right\|_E
\end{align}
for every $\alpha\in \mathbb{C}$.
Hence, using that $x_+ x_- =x_- x_+= 0 $,  we obtain that
\begin{align*}
\left\|\alpha   b_1+\alpha (n+1) b_2\right\|_F &= \left\|T\left(\alpha  x_+ + \alpha(n+1)x_- \right)\right\|_F  =\left\|\alpha   x_+ + \alpha (n+1) x_-\right\|_E \\
&=\left\|\left|\alpha x_+ -\alpha (n+1) x_-\right|\right\|_E =\left\|\alpha  x_+ -\alpha(n+1) x_-\right\|_E\\
& =\left\| \alpha   b_1 - \alpha (n+1) b_2 \right\|_F \stackrel{\eqref{xk77}}{\le}\left\| \alpha  x\right\|_E.
\end{align*}
Thus,  we obtain  validity of \eqref{xk+-} for all $k\ge 1$.
Therefore, since $\alpha$ is arbitrary, it follows that
$$\|   x_-\|_E \le \left\|\frac1n x_+ + x_- \right\|_E \stackrel{\eqref{xk+-}}{\le}  \left\|\frac1n x \right\|_E\stackrel{\eqref{de:d:3}}{\rightarrow}_n 0,$$
which implies that  $x_- =0$.
That is,  $x\ge 0$.
\end{proof}

Corollary \ref{cor:ontoF} below is the main result of this section.
In contrast with the results in \cite{Y, Sherman2005, Sherman}, Corollary \ref{cor:ontoF} covers the case of noncommutative $L_2$-spaces (see Section \ref{s:l}).
It is very common to assume that    symmetrically (quasi-)normed spaces $E(\cM_1,\tau_1)$ and  $F(\cM_2,\tau_2)$ have order-continuous norms or the Fatou property (see e.g. \cite{KR} and \cite[Section 5.2]{FJ}).
Before proceeding to  Corollary \ref{cor:ontoF}, we present the following proposition, which enables us to get rid of the order continuity of $\Delta$-norms.

\begin{proposition}\label{prop:o}
Assume that $E(\cM_1,\tau_1)$ and $F(\cM_2,\tau_2)$ are  symmetrically $\Delta$-normed spaces.
If  $\left\|\cdot\right\|_F$ is (not necessarily strictly) log-monotone,
then any  order-preserving  isometry $T:E(\cM_1,\tau_1)\stackrel{onto}{\longrightarrow} F(\cM_2,\tau_2)$ is normal.
\end{proposition}
\begin{proof}

Assume that $\{x_\alpha\in E(\cM_1,\tau_1)^+\}$  is a net increasing to $x\in E(\cM_1,\tau_1)^+$.
Assume by contradiction that $y:= \sup T(x_\alpha) < T(x)$.
Since $T$ is a bijection, it follows from Lemma \ref{lemma:order} that
$x-T^{-1}(y) = T^{-1} (T(x) -y) >0$.
However, Lemma \ref{lemma:order} also implies that $T^{-1}(y) \ge x_\alpha$,
 which is a contradiction with the assumption.
\end{proof}

Propositions \ref{prop:o} and \ref{prop:positive} guarantee  that we can use Theorem \ref{cor:BJX}  to obtain the following corollary, which extends a number of existing results (see e.g. \cite[Theorem 3.1]{CMS}, \cite[Theorem 1]{Broise} and  \cite[Theorem 2]{Russo}).
\begin{corollary}\label{cor:ontoF}
Assume that $E(\cM_1,\tau_1)\subset S(\cM_1,\tau_1)$ and $F(\cM_2,\tau_2) \subset \cM_2 ^\Delta$ are symmetrically $\Delta$-normed spaces,
and $F(\cM_2,\tau_2)$ has SLM  $\Delta$-norm $\left\|\cdot\right\|_F$.
If there exists an order-preserving isometry $T:E(\cM_1,\tau_1)  \stackrel{onto}{\longrightarrow } F(\cM_2,\tau_2)$,
then there is a Jordan $*$-isomorphism $J: \cM_1 \longrightarrow \cM_2$ and a positive self-adjoint operator $B$ affiliated with $\cZ(\cM_2)$ such that $T(x) =BJ(x)$ for every $x\in E(\cM_1,\tau_2)\cap \cM_1$.
\end{corollary}
\begin{proof}
Let the Jordan $*$-monomorphism $J$ be defined as in  Theorem \ref{cor:BJX}.
It suffices to prove that
   $J(\cM_1)=\cM_2$.

Recall that  $s(T(x)) =J(s(x))$ for every $0\le x\in E(\cM_1,\tau_1)$ (see Remark  \ref{TXJ1E0})
and   $(\cM_2)_{fin}\subset F(\cM_2,\tau_2)$ (see Proposition \ref{prop:car}).
By Lemma  \ref{lemma:order}, for every $f \in \cP_{fin}(\cM_2)$, there exists $0\le x\in E(\cM_1,\tau_1)$ such that $T(x)=f$.
Hence, $f= s(f)=s(T(x)) =J(s(x))$.
This implies that $J(\cM_1)$ contains all $\tau_2$-finite projections in $\cM_2$.
Since $J(\cM_1)$ is a weakly closed $*$-subalgebra of $\cM_2$ (see Remark \ref{isotoJM}) and the latter is semifinite von Neumann algebra,
it follows that $J(\cM_1)=\cM_2$.

Recall that the spectral projections of $B$ commute  with every $J(x)$, $x\in \cM_1$ (see Proposition \ref{propEfinandcom}).
Since $J$ is a surjective, it follows that $B$ is affiliated with $\cZ(\cM_2)$ (see e.g. \cite[Chapter II, Proposition 1.4]{DPS}).
\end{proof}
The following corollary is an extension of \cite[Theorem 2]{Russo}.
\begin{corollary}
Suppose that the assumption of   Corollary \ref{cor:ontoF} are met and, in addition, $\tau_1({\bf 1}_{\cM_1}),\tau_2({\bf 1}_{\cM_2})<\infty$.
If $T({\bf 1}_{\cM_1})={\bf 1}_{\cM_2}$, then $T$  is a Jordan $*$-isomorphism $ \cM_1 $ onto $\cM_2$.
\end{corollary}

\begin{corollary}
Let $(\cM_1,\tau_1)$ and $(\cM_2,\tau_2)$ be two semifinite von Neumann algebras.
If there exists an   order-preserving surjective isometry   $T:E(\cM_1,\tau_1) \rightarrow F(\cM_2,\tau_2)$ for some
 symmetrically $\Delta$-normed spaces $ E(\cM_1,\tau_1)$ and $F(\cM_2,\tau_2)$ ($\left\|\cdot\right\|_F$ is an  SLM $\Delta$-norm),
then
  $\cM_1$ and $\cM_2$ are  Jordan $*$-isomorphic.
\end{corollary}

When  $1\le p<\infty$ and $(\cM,\tau)$ is a finite  factor,
it is shown in \cite[Corollary 1]{Russo} that every order-preserving $L_p$-isometry from $T :\cM \stackrel{onto}{\longrightarrow} \cM$
 is indeed  a   $*$-isomorphism or $*$-anti-isomorphism (see also \cite[Theorem 1]{Katavolos2}).
The following corollary is a semifinite version of \cite[Corollary 1]{Russo} with significant extension.
\begin{corollary}
Suppose that the assumption of   Corollary \ref{cor:ontoF} are met and, in addition, $\cM_2$ is a factor.
Then, there is a constant $\alpha >0$ and a   $*$-isomorphism or a $*$-anti-isomorphism  $J: \cM_1 \longrightarrow \cM_2$ such that $T(x) = \alpha  J(x)$ for every $x\in E(\cM_1,\tau_2)\cap \cM_1$ and $T({\bf 1}_{\cM_1}) =\alpha {\bf 1}_{\cM_2}$.
\end{corollary}

\section{Order-preserving isometries into  Lorentz spaces}\label{s:l}
It is known (see e.g. \cite{KPS,LSZ,DP2}) that every symmetrically normed operator  space is a subspace of $L_1(\cM,\tau)+\cM$.
Indeed, in the $\Delta$-normed setting,
the  $\cL_{\log}(\cM,\tau)+\cM$
 plays a similar role   as $L_1(\cM,\tau)+\cM$ does in the normed case.
In this section, we show that the class of SLM $\Delta$-normed space embraces a wide class  of symmetrically $\Delta$-normed spaces   used in analysis.

Let $p\in (0,\infty)$ and let $w$ be a weight (that is, a non-negative  measurable function on $(0,\infty)$ that is not identically zero).
The Lorentz space $\Lambda^p_w(0,\infty)$ is defined by
$$\left\{f\in S(0,\infty)\mid ~ \|f\|_{\Lambda_w^p}:= \left( \int_0^\infty \mu(t;f)^p w(t) dt \right)^{1/p}<\infty \right\}.$$
For a given weight $w$, we define  $W(t):=\int_0^t w(s)ds$.
We always assume that $W(t)>0$ for every $t\in (0,\infty)$.
It is shown in \cite{CKMP} that $\Lambda_w^p(0,\infty)$ is a linear space if and only if  $W$ satisfies the $\Delta_2$-condition, i.e., $W(2t)\le CW(t)$ for some $C>1$ and  all $t>0$.
Moreover, $\left\|\cdot\right\|_{\Lambda_w^p}$ is a complete quasi-norm    \cite{CS,KM} .
It is known that
$\left\|\cdot\right\|_{\Lambda _w^p}$ is order continuous if and only if $W(\infty)=\infty$ (see e.g. \cite{KM}).
We define
$$\Lambda_w^p(\cM,\tau) := \{x\in S(\cM,\tau):~ \mu(x)\in \Lambda_w^p(0,\infty)\}. $$
In particular, $\Lambda_w^p(\cM,\tau)$ is a quasi-Banach space equipped with quasi-norm $\|X\|_{ \Lambda_w^p}=\|\mu(X)\|_{ \Lambda_w^p}$, $X\in  \Lambda_w^p(\cM,\tau)$ \cite{Sukochev,HLS2017}.

Assume that $w$ is a \textbf{strictly positive decreasing} function on $(0,\infty)$ such that $W(\infty)=\infty$.
Then, $W$ satisfies the $\Delta_2$-condition.
 Moreover,  Proposition \ref{prop:ococ} implies that $\Lambda_w^p(\cM,\tau)$ has order continuous \mbox{(quasi-)}norm and therefore, $\Lambda_w^p(\cM,\tau)\subset \cM^\Delta \subset  S_0(\cM,\tau)$.
In this section, we show that all $\Lambda_w^p(\cM,\tau)$ has SLM quasi-norms.

We note that if there is an isometry $ T$ from a $\Delta$-normed symmetric space  $E(\cM_1,\tau_1)$ into $\Lambda^p_w(\cM,\tau)$, then $E(\cM_1,\tau_1)$ must be quasi-normed.
Indeed, for every $X\in E(\cM_1,\tau_1)$ and $\lambda\in \mathbb{C}$, we have
$$\|\lambda x\|_E = \|T(\lambda x)\|_{\Lambda^p_w} = |\lambda| \|T(x)\|_{\Lambda^p_w} =|\lambda|\|x\|_E,$$
which implies that $E(\cM_1,\tau_1)$ is quasi-normed.
Moreover, if this isometry is surjective, then $E(\cM_1,\tau_1)$ is a quasi-Banach space.

\begin{remark}\label{remark:ae}
Let $x\in S(\cM,\tau)$.
Assume that $w$ is a  strictly positive decreasing function on $(0,\infty)$.
It is easy to see that  $\mu(x)^p w$ and $\mu(\mu(x)^p w)$ are equimeasurable (see e.g. \cite[Chapter III, Section 1]{LSZ}).
Since $\mu(x)^p w$ is a  decreasing function and $\mu(\mu(x)^p w)$ is right-continuous, it is easy to see  that
 $\mu(x)^p w= \mu(\mu(x)^p w)$ a.e..
\end{remark}

The following result is an easy consequence of Corollary \ref{prop:e}.
\begin{corollary}\label{cor:pmaj}Assume that $w$ is a strictly positive decreasing function on $(0,\infty)$.
Let  $ a,b\in \Lambda^p_w(\cM,\tau)$, $p\in (0,\infty)$.
If  $b \prec\prec _{\log} a $,
then $w(t)\mu(t;b)^p \prec \prec w(t)\mu(t;a)^p$.
In particular, $\|b\|_{ \Lambda^p_w}\le \|a\|_{ \Lambda^p_w}$.
\end{corollary}
\begin{proof}
It follows from $ \mu(b)\prec \prec_{\log} \mu(a)$ that
\begin{align*}
\int_0^t \log (w(t)^{1/p}\mu(t;b))dt =\int_0^t \left(\log w(t)^{1/p}+\log \mu(t;b)\right)dt \\
 \le  \int_0^t \left( \log w(t)^{1/p}+\log \mu(t;a)\right)dt =\int_0^t \log \left(w(t)^{1/p}\mu(t;a)\right)dt .
\end{align*}
Thus,  Corollary \ref{prop:e} together  with Remark \ref{remark:ae} implies that
$w(t)\mu(t;b)^p \prec \prec w(t)\mu(t;a)^p $.
\end{proof}
Recall the definition of strictly  $K$-monotone norms defined in Section \ref{s:p}.
The following lemma is an easy consequence of the strict $K$-monotonicity of $L_2$-norm,
showing  that $L_p(\cM,\tau)$ is SLM quasi-normed for every $p\in(0,\infty)$.
\begin{lemma}\label{lemma:1.6}
Let $a,b\in L_p(\cM,\tau)$, $0<p<\infty$, be such that $b \prec \prec_{\log} a$.
If  $\left\|a\right\|_p =\left\|b\right\|_p$, then $\mu(b) =\mu(a )$.
\end{lemma}
\begin{proof}
Assume that $\mu(b) \ne \mu(a)$, i.e., $\mu(b)^{p/2} \ne \mu(a)^{p/2}$.
Corollary \ref{prop:e} implies that $\mu(b)^{p/2} \prec\prec \mu(a)^{p/2}$.
Since $\mu(a)^{p/2},\mu(b)^{p/2}\in L_2(0,\infty)$
 and the $L_2$-norm $\left\| \cdot \right\|_2$ is strictly $K$-monotone (see e.g. \cite[Section 5]{SV} or \cite{CDSS}), it follows that
$$ \left\|b\right\|_p^p=\left\|\mu(b)\right \|_p^p = \left\|\mu(b)^{p/2} \right\|_2 ^2< \left\|\mu(a)^{p/2} \right\|_2^2=\left\|\mu(a) \right\|_p^p=\left\|a\right\|_p^p ,$$
which is a contradiction, that is, $\mu(b) =\mu(a)$.
\end{proof}

The following result is an easy consequence of Lemma \ref{lemma:1.6}, showing  that every $\left\|\cdot \right\|_{\Lambda _w^p }$  is an  SLM quasi-norm.
\begin{theorem}\label{cor:mu=Lo}
Assume that  $0<p<\infty$ and  $w$ is a strictly  positive decreasing function on $(0,\infty)$.
Let  $a,b\in \Lambda^p_w(\cM,\tau)$ be such that $b \prec \prec_{\log} a$.
If   $\left\|a \right\|_{\Lambda^p_w} = \left\|b\right\|_{\Lambda^p_w}$, then $\mu(b) =\mu(a)$.
In particular, $\left\|\cdot\right\|_{\Lambda _w^p }$  is an  SLM quasi-norm.
\end{theorem}
\begin{proof}
Since $\mu(a),\mu(b)\in \Lambda^p_w(0,\infty)$, it follows that $ \mu(a)^p w, \mu(b)^p w \in L_1(0,\infty)$.
It follows from Remark \ref{remark:ae} that
\begin{align*}
 \int_0^s \log  \mu(t;\mu(b)^p w)dt &= \int_0^s \log \mu(t;b)^p w(t)  dt =\int_0^s p\log \mu(t;b)  dt+\int_0^s \log  w(t)  dt \\
&\le  \int_0^s p\log \mu(t;a) dt+\int_0^s \log  w(t) dt \\
&= \int_0^s \log \mu(t;a)^p w(t) dt=  \int_0^s \log  \mu(t;\mu(a)^p w)dt
\end{align*}
and, by the assumption, we have
$$\|\mu(b)^p w\|_1  = \int_0^\infty   \mu(t;\mu(b)^p w)dt =  \int_0^\infty    \mu(t;\mu(a)^p w)dt=\|\mu(a)^p w\|_1.$$
By Lemma \ref{lemma:1.6}, we have $ \mu(\mu(b)^p w) = \mu(\mu(a)^p w)$, which implies that $\mu(b)^p w = \mu(a)^p w$ a.e. (see Remark \ref{remark:ae}).
Since $w$ is a strictly positive function on  $(0,\infty)$, it follows from the right-continuity of $\mu(a)$ and $\mu(b)$ that $\mu(a) =\mu(b)$,
which together with
 Corollary \ref{cor:pmaj} implies that $\left\|\cdot\right\|_{\Lambda _w^p }$ is an SLM $\Delta$-norm.
\end{proof}

Recall that $\Lambda^p_w(\cM,\tau)\subset \cM^\Delta$ and $\left\|\cdot\right\|_{\Lambda^p_w}$ is an order continuous $\Delta$-norm whenever $w$ is a strictly positive decreasing function on $(0,\infty)$ such that $W(\infty)=\infty$.
Moreover,
Theorem \ref{cor:mu=Lo} guarantees that all $\Lambda_w^p(\cM,\tau)$ have SLM quasi-norms.
Appealing to  Corollary \ref{cor:4.5} and Corollary \ref{cor:ontoF}, we obtain immediately  the general form  of order-preserving isometries into/onto Lorentz spaces, respectively, which complements the results in \cite[Section 5]{CMS}.


\renewcommand{\baselinestretch}{1}

$\\$
{\bf Acknowledgements}
The authors would like to thank  Jonathan Arazy, Vladimir Chilin, Jan Hamhalter,  Anna Kaminska and Lajos Molnar for useful
comments on the existing literature,  and Galina Levitina for  helpful discussions.

The first author acknowledges the support of University International Postgraduate Award (UIPA).
The   second   author was supported by the Australian Research Council.
The third author was partly funded by a UNSW Scientia Fellowship.


\end{document}